\documentclass[11pt]{article}
\usepackage{amsmath,amssymb,amsthm}
\usepackage{graphicx,enumerate}
\usepackage{pdfsync,epstopdf,comment,url}
\usepackage[x11names,usenames,dvipsnames,svgnames,table,rgb]{xcolor}
\numberwithin{equation}{section}
\newtheorem{theorem}{Theorem}[section]

\newtheorem{proposition}{Proposition}[section]
\newtheorem{lemma}{Lemma}[section]
\theoremstyle{definition}
\newtheorem{remark}{Remark}[section]
\newtheorem{definition}{Definition}[section]
\newtheorem{example}{Example}[section]
\newtheorem{assumption}{Assumption}

\newcommand{\AAA}{{\cal A}}

\newcommand{\s}{\sigma}

\newcommand{\FF}{{\cal F}}
\newcommand{\EE}{{\cal E}}

\newcommand{\ie}{{\em i.e.}, }
\newcommand{\eg}{{\em e.g.}, }
\newcommand{\cf}{\emph{cf.\ }}
\newcommand{\ksta}{k_{\footnotesize \rm statio}}
\newcommand{\kinfo}{k_{\footnotesize \rm info}}

\newcommand{\nn}{\mathbb{N}} 
\newcommand{\rr}{\mathbb{R}} 
\newcommand{\R}{\mathbb{R}} 
\newcommand{\Zone}{\mathcal{Z}} 
\newcommand{\real}{\mathbb{R}} 
\newcommand{\Lag}{\mathcal{L}} 
\newcommand{\Laug}{\mathcal{L}^{\sharp}} 
\newcommand{\norm}[1]{\left\Vert {#1} \right\Vert} 
\newcommand{\erl}{\left(-\infty , +\infty\right]} 
\newcommand{\dom}[1]{\mathrm{dom}\,{#1}} 
\newcommand{\rank}{\mathrm{rank}}

\newcommand{\crit}[1]{\mathrm{crit}\,{#1}}

\newcommand{\dist}{\mathrm{dist}} 
\newcommand{\act}[1]{\left\langle {#1} \right\rangle} 
\newcommand{\seq}[2]{\left\{{#1}_{{#2}}\right\}_{{#2} \in \mathbb{N}}}
\newcommand{\Seq}[2]{\left\{{#1}^{{#2}}\right\}_{{#2} \in \mathbb{N}}}
\newcommand{\argmin}{\mathrm{argmin}}

\definecolor{lblue}{rgb}{0.8,0.85,1.00}
\definecolor{myblue}{rgb}{.8, .8, 1}
  
\title{Nonconvex Lagrangian-Based Optimization: Monitoring Schemes and Global Convergence}

\author{J\'{e}r\^{o}me Bolte\footnote{TSE  (Universit\'{e} Toulouse I), Manufacture des Tabacs, 21 all\'{e}e de Brienne, 31015 Toulouse, France. E-mail: jerome.bolte@tse-fr.eu.} \and Shoham Sabach\footnote{Faculty of Industrial Engineering, The Technion, Haifa, 32000, Israel. E-mail: ssabach@ie.technion.ac.il.}\and Marc Teboulle\footnote{School of Mathematical Sciences, Tel-Aviv University, Ramat-Aviv 69978, Israel. E-mail: teboulle@post.tau.ac.il.}}
\date{{\em Accepted for publication in ``Mathematics of Operations Research", August 27, 2017}}

\begin{document}
\maketitle

\begin{abstract}
We introduce a novel approach addressing global analysis of a difficult class of nonconvex-nonsmooth optimization problems within the important framework of Lagrangian-based methods. This genuine nonlinear class captures many problems in modern disparate fields of applications. It features complex geometries, qualification conditions, and other regularity properties do not hold everywhere. To address these issues we work along several research lines to develop an original general Lagrangian methodology which can deal, all at once, with the above obstacles. A first innovative feature of our approach is to introduce the concept of Lagrangian sequences for a broad class of algorithms. Central to this methodology is the idea of turning an arbitrary descent method into a multiplier method. Secondly, we provide these methods with a transitional regime allowing us to identify in finitely many steps a zone where we can tune the step-sizes of the algorithm for the final converging regime. Then, despite the min-max nature of Lagrangian methods, using an original Lyapunov method we prove that each bounded sequence generated by the resulting monitoring schemes are globally convergent to a critical point for some fundamental Lagrangian-based methods in the broad semialgebraic setting, which to the best of our knowledge, are the first of this kind.
\end{abstract}

\section{Introduction.}
	Consider the following nonconvex and nonlinear composite minimization problem
	\begin{equation*}
		\text{(CM)} \qquad \mbox{minimize} \left\{ f\left(x\right) \equiv f_{0}\left(x\right) + h\left(F
		\left(x\right)\right) : \, x \in \rr^{n} \right\},
	\end{equation*}
	where
	\begin{itemize}
		\item $f_{0} : \rr^{n} \rightarrow \rr$ is a continuously differentiable function.
 		\item $F : \rr^{n} \rightarrow \rr^{m}$ ($m \leq n$) is a continuously differentiable mapping
 			defined by
 			\begin{equation*}
 				F\left(x\right) := \left(f_{1}\left(x\right) , f_{2}\left(x\right) , \ldots , f_{m}
 				\left(x\right)\right).
 			\end{equation*}
  		\item $h : \rr^{m} \to \erl$ is a proper and lower semi-continuous (lsc) function.
	\end{itemize}
	The structure of the composite model (CM) offers extreme versatility over the traditional nonlinear
	programming formulation. The smooth assumptions are in the mapping $F$ and the function $f_{0}$, while
	constraints, penalties and nonconvex/nonsmooth terms can be handled by the nonconvex and nonsmooth
	function $h$. The composite structure allows to beneficially model a given problem and exploit data
	information, and essentially captures most optimization problems. This is illustrated below in Section
	\ref{ssec:examples}.
\medskip

	The main objective of this paper is to layout the main theoretical tools to achieve a deep understanding
	of augmented Lagrangian based methods and their fundamental properties in the nonconvex setting
	described by model (CM).
\medskip

	The Augmented Lagrangian (AL) methodology has a long history which can be traced back to the works of
	Hestenes \cite{H1969}, Powell \cite{P1969} and Haarhoff and Buys \cite{HB1970} with the so-called
	multipliers method for problems with equality constraints. The AL algorithmic framework was a major
	breakthrough in nonlinear optimization providing the ground to fundamental algorithms and applications
	which have been extensively studied in the literature for various classes of problems. For classical
	results  on the subject including many key results, extensions and closely related schemes such as the
	Proximal Methods of Multipliers (PMM) \cite{R1976} and the Alternating Direction of Multipliers (ADM)
	\cite{FG1983,GLT1989}, we refer the reader to the monographs of Bertsekas \cite{B1982-B} and
	Bertsekas-Tsitsiklis \cite{BT1989-B} and references therein.
\medskip

	Recently, there has been an intensive renewed interest in augmented Lagrangian based methods, and in
	particular within the ADM scheme. This surge of interest is mainly due to the emergence of new and
	modern applications arising in a broad diversity of applications areas such as signal processing, sparse
	approximation in data analysis and machine learning. These problems share particular structures which
	often adapt well to ADM and lead to computationally attractive schemes. A typical prototype which has
	been extensively studied is when all the data is {\em convex} with $F$ being a {\em linear} mapping, and
	problem (CM) reduces to the {\em convex linear composite problem}:
	\begin{equation*}
		\text{\rm (CM-L)} \qquad \mbox{minimize} \left\{ f_{0}\left(x\right) + h\left(Fx\right) : \, x \in
		\rr^{n} \right\}.
	\end{equation*}
	The recent literature on ADM for this convex problem is voluminous and clearly it is not the purpose of
	this paper to review it here. See, for instance, the recent work \cite{ST2014} for an account of old and
	new results on the convergence analysis of various augmented Lagrangian schemes, as well as many
	relevant references to earlier works and to more modern and recent contributions in the convex setting.
\medskip

	This work is a complete departure from the classical convex linear composite model. Indeed, in many of
	the modern applications alluded above the optimization model turns out to be not only nonsmooth but also
	includes inherent nonlinearities which the nonlinear composite model (CM) conveniently captures.
	Unfortunately, while as just mentioned, the analysis of Lagrangian based methods has been extensively
	studied in the convex case, the situation in the nonconvex setting is far from being well understood,
	and global analysis of Lagrangian methods for the general model (CM) remains scarce. In fact, only very
	recently some progress has been initiated in the nonconvex case, but {\em only for the linear} composite
	model (CM-L), see \eg \cite{LP2015} and references therein. Even in the simpler linear composite model,
	the situation is not trivial and the authors in \cite{LP2015} have to rely on various assumptions on the 
	problem's data. Out of studies on the linear composite model, we are not aware of any work attempting to 
	fully understand Lagrangian based methods for the general nonlinear composite model (CM) considered 
	here. The objective of the present work is to address this situation, and to develop the main 
	theoretical tools to achieve a deeper understanding of Lagrangian based methods and their fundamental 
	properties in the nonconvex setting described by the nonlinear composite model (CM).
\smallskip
	
	Before outlining some details on our approach, main contributions and results, we first recall some of
	the major obstacles met in the study of Lagrangian methods by evoking three most salient theoretical
	issues:
	\begin{enumerate}
		\item {\em AL methods are non-feasible methods:} this is due to the very nature of the penalty
			approach used to construct an augmented Lagrangian. As a consequence feasibility issues have to
			be dealt with particular care as they have a direct damaging impact on qualification conditions,
			as explained next.
		\item {\em Failure of qualification conditions}: A major problem with non-feasible methods is that
			qualification conditions must hold in a larger sense in order to allow for the good behavior of
			the algorithm when the current point is far from the feasible set. Yet, for very simple
			constraints, for instance spherical constraints (see Example \ref{ex:sc} and Remark
			\ref{r:failure}), assuming a qualification condition everywhere is not a viable option.
		\item {\em Oscillation issues}: AL methods are particularly well designed to handle problems having
			complex geometry, like for instance nonlinear inequality/equality constrained problems. A
			typical and difficult problem in this context is to tame oscillations of minimizing
			sequences\footnote{Similar difficulties occur in other approaches, see for instance,
			\cite{BP2015} for an illustration in the context of sequentially convex programming approaches,
			and \cite{AUS2015} in the context of an exact penalty approach.}. Moreover, AL methods are of
			{\em min-max dynamics} and thus, by nature, the values taken by the augmented Lagrangian
			function alternatively increase and decrease even if the sequence eventually converges. This
			oscillatory behavior makes the use and the design of Lyapunov functions particularly difficult.
	\end{enumerate}
\medskip

	One of the goals of this paper is to provide the reader with an original general Lagrangian methodology
	which can deal, all at once, with the above obstacles under general and mild assumptions on the
	problem's data. Let us briefly outline our exact contributions now.
\medskip

	The first innovative feature of our approach is to introduce and to study a broad class of algorithms
	through sequences that we call {\em Lagrangian sequences}. At the heart of this methodology is the idea
	of turning an arbitrary descent method into a multiplier method. The rationale is simple, once a method
	or mechanism is chosen, it is implemented on the primal variable(s) of the augmented Lagrangian, while
	the multiplier variable is updated in the classical and straightforward fashion. An illustrative but
	very informative instance of this approach is the famous proximal method of multipliers (PMM) alluded
	above which is modeled through an augmented Lagrangian with an added proximal term and consists of
	performing a proximal step on the primal variable while the multiplier is updated as in the classical AL
	method.
\medskip

	Based on the above methodology, we proceed and describe how we address the three points evoked above.
\medskip

	To circumvent the qualification failures and the lack of knowledge of fundamental constants, we
	introduce the notion of {\em information zone}. It is a subset of the space containing the feasible set
	and on which Lipschitz continuity and qualification conditions are known to hold and are quantifiable by
	simple real numbers (Lipschitz constants and regularity modulus). Then we provide our methodology with
	an {\em adaptive regime} that aims at detecting this zone and at forcing the iterates to stay within the
	zone. The detection of the zone is made by tuning dynamically the penalization parameter of the
	augmented Lagrangian at a sufficiently high value. This approach is shown to identify the zone in
	finitely many steps and deals thus with points 1 and 2.
\medskip

	Once the information zone is found, another crucial issue remains to address: rule out oscillations to
	ensure descent properties of the method, this is point 3 above. This is done by using once more the
	adaptive idea to detect an adequate Lyapunov function. At a technical level this function is
	nonincreasing but the rate of decrease is only controlled for one block of the primal sequence which is
	a departure from classical analysis.
\medskip

	The proposed novel approach and theoretical analysis developed in Sections \ref{Sec:LagCM} to
	\ref{Sec:proofs} allow us to eliminate the difficulties evoked above and to derive a generic Adaptive
	Lagrangian Based mUltiplier Method (ALBUM) for tackling the general nonconvex and nonlinear composite
	model (CM) which encompasses fundamental Lagrangian methods. This paves the way to derive convergence
	results, and in particular, global convergence results to a critical point of problem (CM) with
	semialgebraic data, by relying on the nonsmooth Kurdyka-{\L}ojasiewicz (KL) inequality
	\cite{L1963, K1998, BDL2006}. The potential of our results is demonstrated through the study of two 
	major Lagrangian schemes whose convergence was never analyzed in the proposed general setting: the 
	proximal multiplier method  and the proximal alternating direction of multipliers scheme, this is done 
	in Section \ref{sec:variants} where we also consider some additional interesting variants. We end the 
	introduction with some examples illustrating the versatility of model (CM).

\subsection{Examples of model (CM).} \label{ssec:examples}
	Below we give some examples which exhibit the versatility of model (CM). The first example describes 
	various well-known and classical models in the nonlinear optimization literature, while the remaining 
	four examples describe models arising in some recent applications.
\medskip

	\begin{example}[Nonlinear programming] \label{E:NLP}
		The standard nonlinear program with equality and inequality constraints:
		\begin{equation*}
			\mbox{(NLP)} \qquad \inf_{x \in \rr^{n}} \left\{ f_{0}\left(x\right) : \, f_{i}\left(x\right)
			\leq 0, \, i = 1 , 2 , \ldots , p, \,\, f_{i}\left(x\right) = 0, \, i = p + 1 , p + 2 , \ldots ,
			m \right\},
		\end{equation*}
		can be reformulated through the composite model (CM) by defining the separable model function 	
		$h\left(u\right) := \sum_{i = 1}^{m} h_{i}\left(u_{i}\right)$, where
		\begin{equation*}
			h_{i}\left(u_{i}\right) = i_{(-\infty , 0]}\left(u_{i}\right), \, i = 1 , 2 , \ldots , p, \quad
			\text{and} \quad h_{i}\left(u_{i}\right) = i_{\{0\}}\left(u_{i}\right), \, i = p + 1 , p + 2 ,
			\ldots , m.
		\end{equation*}
		{\em Lagrangians and Smooth penalties.} The standard Lagrangian associated to (NLP) as well as
		linear and quadratic penalty terms can easily be reformulated through model (CM) with a separable
		model function $h$ and an adequate choice of $h_{i}$, $i = 1 , 2 , \ldots , m$. For instance with
		$h_{i}\left(u_{i}\right) = y_{i}u_{i}$, $i = 1 , 2 , \ldots , m$, the standard Lagrangian of problem
		(NLP) is recovered. Likewise the usual penalized counterpart of the problem (NLP) given by
 		\begin{equation*}
 			\mbox{(P-NLP)} \qquad \inf \left\{ f_{0}\left(x\right) + \sum_{i = 1}^{p} \mu_{i}\max \left\{ 0
 			, f_{i}\left(x\right) \right\}^{2} + \sum_{i = p + 1}^{m} \mu_{i}\left|f_{i}\left(x\right)
 			\right|^{2} \right\}, \, (\mu_{i} > 0),
		\end{equation*}
		is recovered through model (CM) with the obvious choices
		\begin{equation*}
			h_{i}\left(u_{i}\right) = \mu_{i}\max \left\{0 , u_{i} \right\}^{2}, \, i = 1 , 2 , \ldots , p ,
			\quad \text{and} \quad  h_{i}\left(u_{i}\right) := \mu_{i}\left|u_{i}\right|^{2}, i = p + 1 , p
			+ 2 , \ldots ,m.
		\end{equation*}
		Obviously, the classical augmented Lagrangian itself for NLP can easily be recovered from model (CM)
		as well, with an adequate piecewise quadratic choice of $h_{i}$, $i = 1 , 2 , \ldots , m$, for
		the inequality constraints.
\smallskip
	
		{\em Nonsmooth and nonseparable $h$.} A classical nonsmooth model is the  $\ell_{1}$-norm penalized
		problem for equality constraints ($p \equiv 0$ in (NLP)) given by
		\begin{equation*}
			\inf_{x \in \rr^{n}} \left\{ f_{0}\left(x\right) + \sum_{i = 1}^{m} w_{i}\left|f_{i}\left(x
			\right)\right| \right\},
		\end{equation*}
 		which is covered by model (CM) with $h_{i}\left(u_{i}\right) := w_{i}\left|u_{i}\right|$ for some
 		$w_{i} > 0$, $i = 1 , 2 , \ldots , m$.
\smallskip

 		{\em Nonseparable nonsmooth: mini-max problems.} Let $f_{0} \equiv 0$ and $h\left(u\right) := \max
 		\left\{ u_{1} , u_{2} , \ldots , u_{m} \right\}$. Then, model (CM) produces the classical nonlinear
 		mini-max problem
		\begin{equation*}
			\inf_{x \in \rr^{n}} \max_{1 \leq i \leq m} f_{i}\left(x\right).
		\end{equation*}
	\end{example}
	The above example exhibit the versatility of model (CM) for traditional NLP. In all these examples $h$
	was convex. We now give three examples with {\em nonconvex} $h$ which include a broad variety of
	fundamental problems arising in applications.
	\begin{example}[Sparsity constrained problems]
		These problems arise in many areas of applications, for example, compressive sensing and machine
		learning see \eg \cite{SNW11}. A basic model (see \cite{BE2013}) reads
		\begin{equation*}
			\min \left\{ f\left(x\right) : \, \norm{x}_{0} \leq s \right\},
		\end{equation*}
		where $\norm{\cdot}_{0}$ stands for the usual counting function, \ie the number of nonzero
		coordinates of $x$, $s > 0$ is the desired sparsity level, and $f$ can be any smooth fidelity
		criterion (\eg least squares). Let $S := \left\{ x : \, \norm{x}_{0} \leq s \right\}$. Then, the
		above problem is a special case of model (CM) with $f_{0}\left(x\right) \equiv f\left(x\right)$, $F
		\left(x\right) \equiv x$ and $h$ is the nonconvex function described by the indicator of the closed
		set $S$, \ie $h\left(u\right) \equiv i_{S}\left(u\right)$.
\medskip

		Matrix rank minimization problems can be similarly formulated in the space of symmetric matrices
		using a constraint of the form $\rank(x) \leq s$.
\medskip
		
		Moreover, nonconvex {\em penalized approximations} of the following form have also been considered
		and found useful (see, \eg \cite{LT14} and references therein)
		\begin{equation*}
			\min \left\{ f\left(x\right) + \rho\sum_{i = 1}^{n} \varphi\left(\left|x_{i}\right|\right) \,
			x \in \rr^{n} \right\}, \quad (\rho > 0 \; \mbox{is a penalty parameter}),
		\end{equation*}
		where $\varphi$ is a concave (increasing) function on $\rr$ used to approximate the $l_{0}$-quasi
		norm. A typical example is the $l_{p}$-quasi norm with $\varphi\left(t\right) := t^{p}$, $0 < p <
		1$, and model (CM) covers this case as well, with an obvious identification for the nonconvex
		function $h$.
	\end{example}
	\begin{example}[Matrix minimization on Stiefel manifolds] \label{ex:sc}
		Optimization problems with matrix orthogonality constraints arise in many applications of science
		and engineering (\eg polynomial optimization, combinatorial optimization, eigenvalue problems,
		sparse PCA, matrix rank minimization, etc., \cite{EAS1998}). A basic problem reads as:
		\begin{equation*}
			\text{(O)} \quad \min \left\{ \Psi\left(X\right) : \, X^{T}X = I, \, X \in \rr^{n \times p}
			\right\},
		\end{equation*}
		where $\Psi : \rr^{n \times p} \rightarrow \rr$ is a smooth function (often quadratic), and $I$
		stands for the $p \times p$ identity matrix. The feasible set ${\cal S}_{n,p} := \left\{ X \in
		\rr^{n \times p} : \, X^{T}X = I \right\}$ is known as the \textit{Stiefel manifold}, which for $p =
		1$ reduces to the unit-sphere manifold ${\cal S}_{n,1} \equiv {\cal S}_{n} = \left\{ x \in \rr^{n} :
		\, \norm{x}_{2} = 1 \right\}$. Clearly, with $h$ being the nonconvex function described by the
		indicator of the closed set ${\cal S}_{n,p}$, problem (O) can easily be seen as a special case of
		model (CM) with the obvious identification for $f_{0}$ and $F$ in the space of real matrices $\rr^{n
		\times p}$.
	\end{example}
	\begin{example}[Nonconvex feasibility] \label{feas}
		Let $S_{1} , S_{2} , \ldots , S_{p}$ (for $p \geq 2$)  be nonempty and closed subsets of $\rr^{n}$. 
		The nonconvex feasibility problem consists in finding a point in the intersection $\displaystyle
		\cap_{i = 1}^{p} S_{i}$. These type of problems abound in many applications such as phase retrieval, 
		network sensors localizations or protein conformation, see \eg \cite{HL2013} for some recent 
		developments. One standard way to tackle the feasibility problem is simply to reformulate it as an 
		optimization problem:
		\begin{equation*}
			\min \left\{ \frac{1}{2\left(p - 1\right)}\sum_{i = 2}^{p} \norm{x_{1} - x_{i}}^{2} + \sum_{i = 
			1}^{p} i_{S_{i}}\left(x_{i}\right) : \,  \left(x_{1} , x_{2} , \ldots , x_{p}\right) \in \rr^{n 
			\times p} \right\},
		\end{equation*}
		Observe that ${\bar x} \in \displaystyle\cap_{i = 1}^{p} S_{i}$ if and only if the optimal value of 
		the above optimization problem at $\left({\bar x} , {\bar x} , \ldots , {\bar x}\right) \in \rr^{n 
		\times p}$ is zero.
\medskip

		Choosing $\rr^{n \times p}$ as the base space, setting $f_{0}\left(x_{1} , x_{2} , \ldots , x_{p}
		\right) = \left(2\left(p - 1\right)\right)^{-1}\sum_{i = 2}^{p} \norm{x_{1} - x_{i}}^{2}$ (which is 
		obviously a $C^{1 , 1}$ function), $F\left(x_{1} , x_{2} , \ldots , x_{p}\right) = \left(x_{1} , 
		x_{2} , \ldots , x_{p}\right)$ and $h\left(x_{1} , x_{2} , \ldots , x_{p}\right) = \sum_{i = 1}^{p} 
		i_{S_{i}}\left(x_{i}\right)$, we see that the above optimization problem fits our general model 
		(CM).
	\end{example}
\medskip

	\noindent {\bf Notations.} For any vector $w \in \real^{d}$, the standard Euclidean norm is simply 
	denoted by $\norm{w}$. Unless otherwise stated, for the subdifferential operators $\hat\partial$, $
	\partial$ and $\partial^{\infty}$ and other objects coming from variational analysis, we adopt the 
	notations and definitions of the monograph by Rockafellar and Wets \cite{RW1998-B}.

\section{The Lagrangian for nonlinear composite problems.} \label{Sec:LagCM}
	This section outlines the first steps toward the generic algorithm we develop and analyze in this paper.
	We define the augmented Lagrangian associated to problem (CM), basic qualification condition and
	assumptions,  and in particular, we introduce the fundamental and new concept of {\em information zone}
	which play a central role in the forthcoming analysis.

\subsection{Lagrangian and qualification condition.} \label{s:lag}
	In analogy to standard NLP, one can construct a natural Lagrangian for problem (CM) as follows. We first
	reformulate problem (CM) in the equivalent split form:
	\begin{equation*}
		\mbox{(CM)} \qquad \inf_{} \left\{ f_{0}\left(x\right) + h\left(u\right) : \, u = F\left(x\right),
		\, (x , u)\in \rr^{n}\times \R^m \right\}.
	\end{equation*}
	For this abstract equality constrained reformulation, the classical {\em Lagrangian} is defined by $\Lag
	: \rr^{n} \times \rr^{m} \times \rr^{m} \to \erl$ via
	\begin{equation*}
		\Lag\left(x , u , y\right) \equiv f_{0}\left(x\right) + h\left(u\right) + \act{y , F\left(x\right) -
		u}.
	\end{equation*}
	An {\em augmented Lagrangian} is a quadratic penalized version of the Lagrangian:
	\begin{align}
		\Laug_{\rho}\left(x , u , y\right) & := \Lag\left(x , u , y\right) + \frac{\rho}{2}\norm{F\left(x
		\right) - u}^{2} \nonumber \\
		& = f_{0}\left(x\right) + h\left(u\right) + \act{y , F\left(x\right) - u}+ \frac{\rho}{2}\norm{F
		\left(x\right) - u}^{2}, \label{D:AugLAc}
	\end{align}
	where $\rho > 0$ is a penalty parameter.
\medskip

	To ensure the well-posedness of the algorithms to come, throughout this paper we assume:
	\begin{equation} \label{WellPosed}
		\inf_{x , u} \Laug_{\rho}\left(x , u , y\right) > -\infty \,\,\, \text{for any fixed} \,\, y \in
		\rr^{m}.
	\end{equation}
	We assume below that model (CM) satisfies a standard qualification condition which we express in the
	compact form provided by variational analysis \cite[Chapter 10, pp. 428--430]{RW1998-B}. We denote by
	$\nabla F\left(x\right) \in \rr^{m \times n}$ the Jacobian matrix of $F$, whose rows are given by the
	gradient vectors $\left[\nabla f_{i}\left(x\right)\right]_{i = 1}^{m}$.
	\begin{assumption} \label{AssumptionA}
		The following constraint qualification  holds for problem (CM),
		\begin{equation*}
			\mbox{[CQ]} \qquad \nabla F\left(x\right)^{T}y = 0 , \quad y \in \partial^{\infty} h\left(F
			\left(x\right)\right) \, \Longrightarrow \, y = 0.
		\end{equation*}
	\end{assumption}	
	For the classical NLP case, which can be obtained from model (CM) as described in Example \ref{E:NLP},
	the condition [CQ] reduces to the classical Mangasarian-Fromovitz constraint qualification, see \eg
	\cite{B1982-B}.
\medskip
	
	The condition [CQ] is not only essential to provide smoothness and regularity of the constraint set, at
	a technical level, it is also important to provide a chain rule for the objective function of model
	(CM). This allows us to derive the first order necessary conditions for this model.
	\begin{definition}[First order optimality condition] \label{D:Opt}
		Let $F : \rr^{n} \rightarrow \rr^{m}$ be a continuously differentiable mapping, and let $h : \rr^{m}
		\rightarrow \erl$ be a proper and lsc function. If $x$ is a local minimizer of problem {\rm (CM)}
		satisfying Assumption \ref{AssumptionA}, then there exists $y \in \rr^{m}$ such that
		\begin{equation*}
 			\nabla f_{0}\left(x\right) + \nabla F\left(x\right)^{T}y = 0 \quad \mbox{with} \quad  y \in
 			\partial h\left(F\left(x\right)\right).
		\end{equation*}
	\end{definition}
	The set of critical points of a function $\psi$, is denoted by $\crit \psi$. For problem (CM) with the
	objective function $f$, we have	
	\begin{equation} \label{crit-f}
		\crit f = \left\{ x \in \rr^{n} : \, 0 \in \nabla f_{0}\left(x\right) + \nabla F\left(x\right)^{T}
		\partial h\left(F\left(x\right)\right) \right\}.
	\end{equation}
	
\subsection{The information zone.}
	Lagrangian based methods require to handle simultaneously penalty parameters, constants, and
	qualification condition which is a delicate matter. An important aspect of this work is to address these
	issues.
\medskip

	Augmented Lagrangian methods are based on relaxing the classical Lagrangian and therefore by nature
	these are unfeasible methods. Measures of unfeasibility of these methods are naturally connected to the
	``looseness" of the relaxation. The looser is the relaxation, the more unfeasible is the method. Over
	relaxation could even result in absurd behaviors.
\medskip

	The augmented Lagrangian $\Laug_{\rho}$ as given in \eqref{D:AugLAc} is
	\begin{equation*}
		\Laug_{\rho}\left (x , u , y\right) := f_{0}\left(x\right) + h\left(u\right) + \act{y , F\left(x
		\right) - u} + \frac{\rho}{2}\norm{F\left(x\right) - u}^{2}, \,\, \mbox{ with }\rho > 0.
	\end{equation*}
	In this context the looseness/sharpness of the relaxation is embodied within the penalty parameter $\rho
	$ which is used to penalize the constraint $F\left(x\right) = u$ in the augmented Lagrangian
	$\Laug_{\rho}$. At an analytic level this penalty reflects the fact that for a fixed triple $\left(x , u
	, y\right)$ one has
	\begin{equation*}
		\lim_{\rho \rightarrow +\infty} \Laug_{\rho}\left(x , u , y\right) =
		\begin{cases}
			f_{0}\left(x\right) + h\left(F\left(x\right)\right), & \mbox{ if } F\left(x\right) = u, \\
			+\infty, & \text{ otherwise,}
		\end{cases}
	\end{equation*}
	which amounts, in some sense, to the convergence of $\Laug_{\rho}$ to $\Lag$ as $\rho \rightarrow +
	\infty$.	
\medskip
	
	A major drawback of such unfeasible methods, easily guessed from the above, is that they generate points
	that might be out of control in the sense that:
	\begin{itemize}
		\item[--] constraint qualification conditions may fail,
		\item[--] assumptions on the problem's data, such as global Lipschitz constants of the various
			objects involved may become unknown or out of reach.
	\end{itemize}
\medskip

	On the other hand, assuming a global control is very demanding and could be unrealistic in practice.
\medskip

	To remedy these obstacles all at once our approach is twofold: we first define an information zone,
	denoted by $\Zone$, to be a region for which regularity is under control and constants are known. Second
	we provide a generic Lagrangian scheme described below with an extra-adaptive search made to reach the
	information zone\footnote{As we shall see soon the adaptive regime allows also for dynamic adjustment of
	the step-sizes to other geometrical features.}

 	Let $\dom{h} = \left\{ u \in \rr^{m} : \, h\left(u\right) < \infty \right\}$ which is nonempty and
 	closed. Then the feasible set of problem (CM) is defined by
	\begin{equation*}
		\FF = \left\{ x \in \rr^{n} : \, F\left(x\right) \in \dom h \right\}.
	\end{equation*}
	\begin{definition}[Information zone] \label{def:ifoz}
		Given the feasible set $\FF$ for problem (CM), an information zone is a subset $\Zone$ of $\rr^{n}$
		such that there exists ${\bar d} \in \left(0 , +\infty\right]$ for which
		\begin{equation} \label{zone}
			\Zone \supset \left\{ x \in \rr^{n} : \, \dist\left(F\left(x\right) , \dom h\right) \leq {\bar
			d} \, \right\} \supset \FF.
		\end{equation}
	\end{definition}
	The information zone is an enlargement of the feasible set $\FF$. It should be noted that the 
	information zone $\Zone$ depends on the parameter ${\bar d}$. For simplicity of exposition, in the 
	forthcoming section, this dependence is not explicitly mentioned. In the next definition we recall a 
	fundamental and classical regularity assumption (see, \eg Milnor \cite{M31}).
	\begin{definition}[Uniform regularity] \label{def:regul}
		Let $\Omega$ be an open subset of $\rr^{n}$, $F : \Omega \rightarrow \rr^{m}$ be a continuously
		differentiable mapping, and let $S$ be a nonempty subset of $\Omega$. We say that $F$ is uniformly
		regular on $S$ with constant $\gamma > 0$ if the following holds:
		\begin{equation*}
			\norm{\nabla F\left(x\right)^{T}v} \geq \gamma\norm{v}, \,\, \forall \, x \in S, \,\, v \in
			\rr^{m}.
		\end{equation*}
	\end{definition}
	\begin{remark} \label{r:param1}
 		For a given $x \in \Omega$, asserting that
 		\begin{equation*}
 			\gamma(F,x) = \min\left\{ \norm{\nabla F\left(x\right)^{T}v} : \, \norm{v} = 1 \right\},
 		\end{equation*}
 		is nonzero is equivalent to the fact that $\nabla F\left(x\right)$ is surjective or $\nabla F\left(x
 		\right)\nabla F\left(x\right)^{T}$ is positive definite. In nonlinear optimization it is also known
 		as Mangasarian-Fromovitz condition at~$x$. Geometrically it means that the set  $\left\{ y \in U :
 		\, F\left(y\right) = F\left(x\right) \right\}$ is a $C^{1}$ manifold for any small open neighborhood
 		around $x$.
\medskip

		Note also that
		\begin{equation} \label{eigen}
			\gamma \equiv \gamma(F,x) = \sqrt{\lambda_{\min}\Big(\nabla F\left(x\right)\nabla F\left(x
			\right)^{T}\Big)}, 				
		\end{equation}
		where $\lambda_{\min} (A)$ denotes the smallest eigenvalue of a real symmetric matrix $A$.
	\end{remark}

\subsection{Basic assumptions for model (CM).}
	We introduce the following essential assumptions.
	\begin{assumption} \label{AssumptionB}
		Given an information zone $\Zone$, we assume that:
		\begin{itemize}
			\item[$\rm{(i)}$] $F$ is uniformly regular over $\Zone$ with constant $\gamma$,
			\item[$\rm{(ii)}$] $\nabla F$ is $L(F)$ Lipschitz continuous over $\Zone$,
			\item[$\rm{(iii)}$] $\nabla f_{0}$ is $L(f_{0})$ Lipschitz continuous over $\Zone$.
		\end{itemize}
	\end{assumption}
	\begin{remark} \label{r:info}
		\begin{itemize}
			\item[(a)] Naturally, the Lipschitz continuity and the uniform regularity are not required on
				the whole space $\rr^{n}$, but only on the information zone $\Zone$. This is a departure
				from the usual setting.
			\item[(b)] When $\nabla f_{0}$ is known to be Lipschitz continuous on the whole space $\rr^{n}$,
				and the mapping $F$ is assumed to be linear, \ie $F\left(x\right) = Fx$ for some matrix $F
				\in \rr^{n \times m}$ with full row rank, then Assumption \ref{AssumptionB} holds with
				$\Zone \equiv \rr^{n}$ (\ie ${\bar d} = +\infty$) and $FF^{T} \succeq \gamma I_{n}$ where
				$\gamma = \sqrt{\lambda_{\min}(FF^{T})} > 0$. 		
		\end{itemize}
	\end{remark}
\medskip

	Let us illustrate the concept of the information zone on a simple but fundamental and emblematic 
	situations (\cf Example \ref{ex:sc}).
	\begin{example}[Spherical constraints] \label{cex}
		Assume that $F\left(x\right) = \norm{x}^{2}$ and $h = i_{\{ 1 \}}$. For simplicity we also assume
		that $f_{0}$ is globally Lipschitz.
\smallskip

		One has $\nabla F\left(x\right) = 2x$ and thus for a fixed $x$, $\gamma\left(F , x\right) =
		2\norm{x}$. Take $r_{1} \in \left(0 , 1\right)$, and define $\Zone = \left\{ x \in \rr^{n} : \,
		r_{1} \leq \norm{x} \right\}$. We see that $F$ is $2r_{1}$ regular on $\Zone$  and $\nabla F$ is
		$2$-Lipschitz continuous. Hence $\Zone$ can be chosen as an information zone as long as we show
		that \eqref{zone} holds true. Take ${\bar d} = 1 - r_{1}^{2}$, it is easy to check that $\left|
		\norm{x}^{2} - 1\right| \leq {\bar d}$ implies, in particular, that $1 - \norm{x}^{2} \leq 1 - r_{1}
		^{2}$. Note that $0 \notin \Zone$ and that $\rr^{n}$ could not be an acceptable choice for an
		information zone because of the degeneracy of $\nabla F$ at $\left\{ 0 \right\}$. 	
	\end{example}
	\begin{remark}[Systematic failure of global CQ with compact equality constraints] \label{r:failure}
		The preceding example reveals a simple and systematic phenomenon which motivates strongly the use of
		an information zone. Consider a  $C^{1}$ function $F:\rr^{n} \rightarrow \rr$ such that $[F = 0]$ is
		a compact manifold and assume that $\mbox{int } [F\leq 0] = [F<0]$. Then, necessarily there exists
		$x^{\ast}$ such that $\nabla F\left(x^{\ast}\right) = 0$. Indeed, by taking $x^{\ast}$ to be a
		minimizer of $F$ over the compact set $[F \leq 0]$ and since this minimizer lies within $[F < 0]$ it
		follows that $\nabla F\left(x^{\ast}\right) = 0$. This shows that in general, it is not possible, to
		have $\Zone = \rr^{n}$.
	\end{remark}

\section{Adaptive Lagrangian based multiplier method.} \label{Sec:ALBUM}
$ $
\medskip

	\begin{center}
		\fbox{From now on Assumptions \ref{AssumptionA} and \ref{AssumptionB} form our blanket assumptions.}
	\end{center}	
\medskip
	
	As explained previously, difficult obstacles are faced both in the design and the study of Lagrangian
	based methods: lack of descent, and above all, feasibility issues. The adaptive idea we develop here is
	precisely meant to put us in a position where these issues are treated in a dynamical fashion: both the
	information zone and the ``energy functional" $\EE_{\beta}$ which we introduce now come into a play.

\subsection{Lagrangian and a Lyapunov function.}	
	We shall need to work with an  auxiliary function which is very similar to the augmented Lagrangian
	$\Laug_{\rho}$ (defined in \eqref{D:AugLAc}). This is a classical approach often called the ``Lyapunov"
	methodology. It will reveal the optimizing property of the generic Lagrangian scheme we introduce next.
\medskip

	Let $\beta > 0$ and $w \in \rr^{n}$, here we consider the Lyapunov function which is defined by
	\begin{equation} \label{D:Lyap}
		\EE_{\beta}\left(x , u , y , w\right) : = \Laug_{\rho}\left(x , u , y\right) + \beta\norm{x - w}
		^{2}.
	\end{equation}
	Below, we record the relationships between the critical point sets of the three relevant functions $f$,
	$\Laug_{\rho}$ and $\EE_{\beta}$. These relations already suggest the pivotal role to be played by
	$\EE_{\beta}$. Recall that condition [CQ] is always assumed, \ie Assumption \ref{AssumptionA} holds.
	\begin{proposition}[Critical points relationships] \label{P:Crit}
		Let $x \in \rr^{n}$ and $u , y \in \rr^{m}$. The following implications hold:
		\begin{equation*}
			\left(x , u , y , x\right) \in \crit{\EE_{\beta}} \, \Longrightarrow \, \left(x , u , y
			\right) \in \crit{\Laug_{\rho}} \, \Longrightarrow \, x \in \crit{f},
		\end{equation*}
		for all $\beta , \rho > 0$.
	\end{proposition}
	\begin{proof}
		The result follows easily from standard subdifferential calculus rules. Indeed, from the definition
		of $\EE_{\beta}$ (see \eqref{D:Lyap}) we have that $\left(x , u ,  y , w\right) \in
		\crit{\EE_{\beta}}$ if and only if
		\begin{equation} \label{critE}
			\hspace{-0.07in} \left(0 , 0 , 0 , 0\right) \in \left(\nabla_{x} \Laug_{\rho}\left(x , u , y
			\right) + 2\beta\left(x - w\right) , \partial_{u} \Laug_{\rho}\left(x , u , y\right) ,
			\nabla_{y} \Laug_{\rho}\left(x , u , y\right) , 2\beta\left(w - x\right)\right).
		\end{equation}
		On the other hand, using the definition of $\Laug_{\rho}$ (see \eqref{D:AugLAc}) we obtain
		\begin{align}
			\nabla_{x} \Laug_{\rho}\left(x , u , y\right) & = \nabla f_{0}\left(x\right) + \nabla F\left(x
			\right)^{T}\left(y + \rho\left(F\left(x\right) - u\right)\right), \label{critAc1} \\
			\partial_u \Laug_{\rho}\left(x , u , y\right) & = \partial h\left(u\right) + \rho\left(u - F
			\left(x\right)\right) - y, \label{critAc2} \\
			\nabla_{y} \Laug_{\rho}\left(x , u , y\right) & = F\left(x\right) - u. \label{critAc3}
		\end{align}
		Therefore, taking $w = x$ in \eqref{critE}, the first implication in the proposition follows. The
		second implication follows by noticing that with $\left(x , u , y\right) \in \crit{\Laug_{\rho}}$,
		the three relations \eqref{critAc1}, \eqref{critAc2} and \eqref{critAc3} reduce to $0 = \nabla f_{0}
		\left(x\right) + \nabla F\left(x\right)^{T}y$ and $0 \in \partial h\left(F\left(x\right)\right) - y
		$. Hence, using Definition \ref{D:Opt}, we obtain that $x \in \crit{f}$. This complete the proof.
	\end{proof}

\subsection{A generic algorithm: {\bf ALBUM}.} \label{s:algo}
	In order to describe the forthcoming generic scheme, we first need to introduce a primal black-box map
	which governs the mechanism of the global convergence methodology to be developed in Section
	\ref{SSec:Meth}.
	\begin{definition}[Lagrangian algorithmic map] \label{D:LagAlg}
		Consider the optimization model (CM) and its associated augmented Lagrangian $\Laug_{\rho}$ which is
		defined in \eqref{D:AugLAc}. Let $\left(x , u , y\right) \in \real^{n} \times \real^{m} \times
		\real^{m}$ be any given triple. A \emph{primal black-box map} $\AAA_{\rho}$ generates a couple
		$\left(x^{+} , u^{+}\right)$ by
		\begin{equation*}
       		\left(x^{+} , u^{+}\right) \in \AAA_{\rho}\left(x , u , y\right).
		\end{equation*}
		A primal black-box map $\AAA_{\rho}$ is called a \emph{Lagrangian algorithmic map} if there are two
		positive constants $a$ and $b$ such that
		\begin{equation*}
			\mbox{(i)} \quad \frac{a}{2}\norm{x^{+} - x}^{2} + \Laug_{\rho}\left(x^{+} , u^{+} , y\right)
			\leq \Laug_{\rho}\left(x , u , y\right),
		\end{equation*}
		and
		\begin{equation*}
			\hspace{-0.7in} \mbox{(ii)} \hspace{0.2in} \norm{\nabla_{x} \Laug_{\rho}\left(x^{+} , u^{+} , y
			\right)} \leq b\norm{x^{+} - x}.
		\end{equation*}
	\end{definition}
\medskip

	Thus, once we chose the Lagrangian algorithmic map $\AAA_{\rho}$, this choice fully determine the
	constants $a$ and $b$, which play an important role in the generic algorithm outlined below. Note that
	these constants  might depend on the problem's data input (\eg Lipschitz constant, uniform regularity 	
	constant, or/and algorithmic constants, \eg proximal/penalty parameters). We deferred to Section~\ref{sec:variants} for two instances of fundamental  Lagrangian algorithmic maps.
\medskip

	The proposed generic adaptive algorithm aims at forcing $x^{k}$ to enter the information zone, which is
	a minimal requirement if we hope for a good behavior of our unfeasible schemes.
\medskip

    {\center\fbox{\parbox{16cm}{{\bf Adaptive Lagrangian-Based mUltiplier Method -- ALBUM}
		\begin{enumerate}
			\item[1.] Input: $\AAA_{\rho}$ a Lagrangian algorithmic map.
			\item[2.] Initialization: Fix $\delta , \rho_{0} > 0$ and start with any $\left(x^{0} , u^{0} ,
				y^{0}\right) \in \rr^{n} \times \rr^{m} \times \rr^{m}$.		
            \item[3.] For each $k = 0 , 1 , \ldots$ generate a sequence $\left\{ \left(x^{k} , u^{k} ,
            		y^{k}\right) \right\}_{k \in \nn}$ as follows:
                	\begin{enumerate}
                    	\item[3.1.] Primal step
                       		\begin{equation} \label{GenericAdap:PriStep}
                        		\left(x^{k + 1} , u^{k + 1}\right) \in \AAA_{\rho_{k}}\left(x^{k} , u^{k} ,
                        		y^{k}\right).
							\end{equation}
                    	\item[3.2.] Multiplier step
                       		\begin{equation} \label{GenericAdap:MultiStep}
                           		y^{k + 1} = y^{k} + \rho_{k}\left(F\left(x^{k + 1}\right) - u^{k + 1}
                           		\right).
                       		\end{equation}
					\item[3.3.] Adaptive step: choose $\tau \in \left(0 , \frac{a}{2}\right)$ and set
						$\beta_{k} := \frac{b^{2}}{\rho_{k}\gamma}$. If $x^{k + 1} \notin \Zone$ or
						\begin{equation} \label{GenericAdap:AdapStep}
							\tau\norm{x^{k + 1} - x^{k}}^{2} > \EE_{\beta_{k}}\left(x^{k} , u^{k} , y^{k} ,
							x^{k - 1}\right) - \EE_{\beta_{k}}\left(x^{k + 1} , u^{k + 1} , y^{k + 1} ,
							x^{k}\right),
						\end{equation}
						set $\rho_{k + 1} = \rho_{k} + \delta$. Otherwise, set $\rho_{k + 1} = \rho_{k}$.
        			\end{enumerate}
			\end{enumerate}}}}
\vspace{0.2in}

	The relations between $a$, $b$, the penalty parameters sequence $\seq{\rho}{k}$ and other data input
	constants will be made more precise whence we develop our analytic framework in Section \ref{Sec:KeyL}.
\medskip

	We record here a simple consequence which will be useful in our analysis that immediately follows from
	the definitions of $\rho_{k}$ and $\beta_{k}$ (see Step 3.3):
 	\begin{equation} \label{parcb}
  		\rho_{k} \geq \rho_{0} > 0 \quad \text{and} \quad \beta_{k} \leq \beta_{0}, \,\,\, \text{for all} \,
  		\, k \in \nn.
  	\end{equation}
	\begin{remark} \label{r:param2}
		In some cases the penalty parameters $\rho_{k}$, $k \in \nn$, can be adjusted so that Step 3.3
		automatically holds with $\rho_{k} = \rho$ for all $k \in \nn$. In this case the iterations boils
		down to Steps 3.1 and 3.2 only. This will happen for instance in the case when the information zone
		is the whole space, \eg when $F$ is linear (\cf Remark \ref{r:info} and Remark \ref{r:lin} below).
	\end{remark}
	
\subsection{A methodology for Lagrangian based methods.} \label{SSec:Meth}
	First note that, once the input Lagrangian algorithmic map $\AAA_{\rho}$ is chosen, {\bf ALBUM}
	generates a sequence $\Seq{z}{k} := \left\{ \left(x^{k} , u^{k} , y^{k}\right) \right\}_{k \in \nn}$,
	which thanks to Definition \ref{D:LagAlg}, must satisfy the following two conditions
	\begin{itemize}
		\item[{\bf C1}] There exists a positive constant $a$ such that
			\begin{equation*}
				\frac{a}{2}\norm{x^{k + 1} - x^{k}}^{2} + \Laug_{\rho_{k}}\left(x^{k + 1} , u^{k + 1} ,
				y^{k}\right) \leq \Laug_{\rho_{k}}\left(x^{k} , u^{k} , y^{k}\right), \quad \forall \,\, k
				\geq 0.
			\end{equation*}
		\item[{\bf C2}] There exists a positive constant $b$ such that
			\begin{equation*}
				\norm{\nabla_{x} \Laug_{\rho_{k}}\left(x^{k + 1} , u^{k + 1} , y^{k}\right)} \leq b
				\norm{x^{k + 1} - x^{k}}, \quad \forall \,\, k \geq 0.
			\end{equation*}
	\end{itemize}	
	Independently of the algorithmic map $\AAA_{\rho}$ which governs the mechanism of a primal black-box,
	we also need two additional assumptions on the corresponding generated sequence $\Seq{z}{k}$ which we
	record now:
	\begin{itemize}
		\item[{\bf C3}] There exists a positive constant $c$ such that
			\begin{equation*}
				\norm{v^{k + 1}} \leq c\norm{x^{k + 1} - x^{k}}, \quad \forall \,\, k \geq 0,
			\end{equation*}
			for some $v^{k + 1} \in \partial_{u} \Laug_{\rho_{k}}\left(x^{k + 1} , u^{k + 1} , y^{k}\right)
			$.
		\item[{\bf C4}] Let ${\bar u}$ be a limit point of a subsequence $\left\{ u^{k} \right\}_{k \in
			{\cal K}}$ of $\Seq{u}{k}$, then $\limsup_{k \in {\cal K} \subset \nn} h(u^k) \leq h({\bar u})$.
	\end{itemize}
\medskip

	Some comments are now in order. First, note that the proposed methodology, while similar in spirit, is
	fundamentally different from the general methodology recently proposed in \cite{BST2014}, which is
	unfortunately not applicable for {\bf ALBUM}, due to the primal-dual structure of this scheme. In
	particular,
	\begin{itemize}
		\item The first condition {\bf C1} is a {\em partial descent property} on $\Laug_{\rho}\left(\cdot
			\right)$. It pertains to the primal variables $\left(x , u\right)$, since by nature the dual
			variable $y$ is an ``ascent variable". The dissymmetry between $x$ and $u$ in the descent
			condition could be removed by further generalizing our approach. For the sake of simplicity, we
			only consider the case when the quantity of decrease in $x$ is known.
		\item Conditions {\bf C2} and {\bf C3} provide subgradient bounds for $\Laug_{\rho}\left(\cdot
			\right)$ with respect to the primal variables.
		\item The sequential assumption on $h$, that is, condition {\bf C4}, is a minimal and extremely weak
			requirement. This property holds for instance when $h : \dom h \rightarrow \rr$ is continuous.
	\end{itemize}	
\medskip

	From now on, and through the rest of this paper we adopt the following terminology:
\medskip

	\begin{center}
		\fbox{\parbox{11cm}{\center \vspace{-0.15in}A sequence $\Seq{z}{k}$ which is generated by {\bf ALBUM} and satisfies
		conditions {\bf C1}--{\bf C4} is called a \textit{Lagrangian sequence}.}}
	\end{center}
\medskip

	As we shall see soon, many fundamental Lagrangian based methods produce  Lagrangian sequences. This
	allows us to derive convergence results in a unified way for such methods and their variants. We
	postpone the description of these methods to Section \ref{sec:variants}, and we announce next, our main
	convergence results for {\bf ALBUM}, which will be proved in the following sections.

\subsection{Main convergence results for {\bf ALBUM}.} \label{mainconv}
	Our central theoretical contributions on the convergence  of {\bf ALBUM} to a critical point of problem
	(CM) are stated in the following two results.
	\begin{theorem}[Subsequence convergence] \label{T:SubConv}
		Let $\Seq{z}{k}$ be a bounded Lagrangian sequence and let $\left({\bar x} , {\bar u} , {\bar y}
		\right)$ be a limit point of $\Seq{z}{k}$. Then ${\bar x}$ is a critical point of the original
		problem {\rm(CM)}.
	\end{theorem}
	Considering semi-algebraic or definable data, and relying on the so-called nonsmooth KL property
	\cite{BDL2006}, we can rule out oscillatory behaviors and  establish the global convergence of the whole
	sequence.
	\begin{theorem}[Global convergence] \label{T:GlobConv}
		Under the premises of Theorem \ref{T:SubConv}, and assuming that $f_{0}$, $F$ and $h$ are
		semi-algebraic, the whole sequence $\Seq{z}{k}$ converges to a point $\left({\bar x} , {\bar u} ,
		{\bar y}\right)$ such that ${\bar x}$ is a critical point of problem {\rm (CM)}.
	\end{theorem}
	\begin{remark} \label{r:rate}
		\begin{itemize}
			\item[$\rm{(i)}$] Standard arguments show that convergence rates of the sequence $\Seq{z}{k}$ of
				the type $\displaystyle O\left(k^{-s}\right)$ could be established with $s > 0$. We refer to
				the technique in \cite{AB2009}.
			\item[$\rm{(ii)}$] The essential tools for convergence are elementary stability questions and
				the nonsmooth Kurdyka-\L ojasiewicz inequality, and thus semi-algebraicity can be replaced
				by definability in a o-minimal structure on $\R,+,\times$.
		\end{itemize}
	\end{remark}
	
\medskip
	
	The next section develops our analytically framework. We present the main ideas underlying the proposed
	algorithm, the main obstacles that need to be addressed, and the key tools necessary for developing the
	convergence analysis of {\bf ALBUM}.

\section{A key lemma: penalty parameter stabilization.} \label{Sec:KeyL}
	In this section, we establish a central result which is essential in our approach. It asserts that the
	sequence of penalty parameters $\seq{\rho}{k}$ becomes stationary and that the information zone $\Zone$
	is reached within finitely many steps. To establish this result, we provide in a preliminary subsection
	some simple but yet fundamental properties.

\subsection{Fundamental properties of Lagrangian sequences.}
	The first elementary result identifies when an iterate enters the information zone $\Zone$.
	\begin{lemma}[Information lemma] \label{L:Stab}
		Let $\Zone$ be a given information zone. Let $\Seq{z}{k}$ be a Lagrangian sequence and assume that
		the multiplier sequence $\Seq{y}{k}$ is bounded. Then, there exists an index $\kinfo \in \nn$, such
		that $x^{k} \in \Zone$ for all $k \geq \kinfo$.
	\end{lemma}
	\begin{proof}
		We argue by contradiction and assume that $x^{k} \notin \Zone$ for $k \in I$ where $I$ is an
		infinite set. On one hand, by the definition of the information zone $\Zone$, we have for all $k \in
		I$ that
		\begin{equation} \label{L:Stab:1}
			\dist\left(F\left(x^{k}\right) , \dom{h}\right) > \bar{d}.
		\end{equation}
		On the other hand, for all $k \in \nn$ we have
		\begin{align*}
			\dist\left(F\left(x^{k}\right) , \dom{h}\right) & = \inf_{u \in \dom{h}} \norm{u - F\left(x^{k}
			\right)} \\
			& \leq \norm{u^{k} - F\left(x^{k}\right)} \hspace{1.2in} \left[u^{k} \in \dom{h}\right] \\
			& = \frac{1}{\rho_{k - 1}}\norm{y^{k} - y^{k - 1}} \hspace{1.03in}
			\left[\eqref{GenericAdap:MultiStep}\right] \\
			& \leq \frac{M}{\rho_{k - 1}}. \hspace{1.78in}\left[\Seq{y}{k} \;\text{is assumed bounded}
			\right]
		\end{align*}
		By Step 3.3 of the algorithm and the fact that $I$ is an infinite set, it follows that $\rho_{k}
		\rightarrow \infty$ as $k \rightarrow \infty$, thus there exists $\kinfo \in \nn$ such that
		\begin{equation*}
			\dist\left(F(x^{k}), \dom{h}\right) \leq \frac{M}{\rho_{k}} \leq \bar{d}, \quad \forall \,\, k
			\geq \kinfo,
		\end{equation*}
		which obviously contradicts \eqref{L:Stab:1}.
	\end{proof}
	The next result provides an important relation on the sequences $\Seq{x}{k}$ and $\Seq{y}{k}$ produced
	by {\bf ALBUM} and reflects the min-max dynamics at the root of these methods.
	\begin{lemma} \label{L:DescentProperty}
        Let $\Seq{z}{k}$ be a Lagrangian sequence. The following inequality holds true for any $k \geq 0$
        \begin{equation*}
	         \Laug_{\rho_{k}}\left(x^{k + 1} , u^{k + 1} , y^{k+1}\right) - \Laug_{\rho_{k}}\left(x^{k} ,
	         u^{k} , y^{k}\right) \leq \frac{1}{\rho_{k}}\norm{y^{k + 1} - y^{k}}^{2} - \frac{a}{2}
	         \norm{x^{k + 1} - x^{k}}^{2}.
        \end{equation*}
    \end{lemma}
    \begin{proof}
 		From condition {\bf C1},
        \begin{equation} \label{L:DescentProperty:1}
        		\Laug_{\rho_{k}}\left(x^{k + 1} , u^{k + 1} , y^{k}\right) - \Laug_{\rho_{k}}\left(x^{k} , u^{k}
        		, y^{k}\right) \leq -\frac{a}{2}\norm{x^{k + 1} - x^{k}}^{2}.
		\end{equation}
        Using the definition of $\Laug_{\rho}$ (\cf \eqref{D:AugLAc}) we have from
        \eqref{GenericAdap:MultiStep} that
        \begin{align*}
	        	\Laug_{\rho_{k}}\left(x^{k + 1} , u^{k + 1} , y^{k+1}\right) - \Laug_{\rho_{k}}\left(x^{k + 1} ,
	        	u^{k + 1} , y^{k}\right) & = \act{y^{k + 1} - y^{k} , F\left(x^{k + 1}\right) - u^{k + 1}} \\
	        	& = \frac{1}{\rho_{k}}\norm{y^{k + 1} - y^{k}}^{2}.
        \end{align*}
        Adding the latter to \eqref{L:DescentProperty:1} yields the desired result.
    \end{proof}
	The next result relates the evolution of the multiplier sequence $\Seq{y}{k}$ with that of the primal
	sequence $\Seq{x}{k}$.
	\begin{lemma} \label{L:DualSequence}
        Let $\Seq{z}{k}$ be a Lagrangian sequence. Assume that the multiplier sequence $\Seq{y}{k}$ is
        bounded by some $\Lambda > 0$. Then, the following inequality holds true for any $k \geq \kinfo$,
		\begin{equation} \label{L:DualSequence:00}
        		\norm{y^{k + 1} - y^{k}}^{2} \leq d_{1}\norm{x^{k + 1} - x^{k}}^{2} + d_{2}\norm{x^{k} -
        		x^{k - 1}}^{2},
		\end{equation}
		where
		\begin{equation} \label{L:DualSequence:0}
        		d_{1} = \frac{2}{\gamma^2}\left(L(f_{0}) + L(F)\Lambda + b\right)^{2} \quad \text{and} \quad
        		d_{2} = \frac{2b^{2}}{\gamma^{2}}.
		\end{equation}
    \end{lemma}
    \begin{proof}
	    	For convenience, we define
        \begin{equation*}
        		\Delta_{k} := \nabla F\left(x^{k + 1}\right)^{T}y^{k + 1} - \nabla F\left(x^{k}\right)^{T}
        		y^{k}.
        \end{equation*}
       	Then, by Lemma \ref{L:Stab} and Assumption \ref{AssumptionB}(i) and (ii) which warrants that $F$ is
       	uniform regular on $\Zone$ with constant $\gamma$ and $\nabla F$ is Lipschitz continuous on $\Zone$,
       	respectively, it follows for all $k \geq \kinfo$ that
	    	\begin{align}
    			\norm{\Delta_{k}} & = \norm{\nabla F\left(x^{k + 1}\right)^{T}\left(y^{k + 1} - y^{k}\right) +
    			\left(\nabla F\left(x^{k + 1}\right) - \nabla F\left(x^{k}\right)\right)^{T}y^{k}} \nonumber \\
	        & \geq \gamma\norm{y^{k + 1} - y^{k}} - L(F)\Lambda\norm{x^{k + 1} - x^{k}}.
	        \label{L:DualSequence:2}
        \end{align}
    		On the other hand, from the definition of $\Laug_{\rho}$ (see \eqref{D:AugLAc}), we have that
    		\begin{align*}
    			\nabla_{x} \Laug_{\rho_{k}}\left(x^{k + 1} , u^{k + 1} , y^{k}\right) & = \nabla f_{0}\left(x^{k
    			+ 1}\right) + \nabla F\left(x^{k + 1}\right)^{T}\left(y^{k} + \rho_{k}\left(F\left(x^{k + 1}
    			\right) - u^{k + 1}\right)\right) \nonumber \\
    			& = \nabla f_{0}\left(x^{k + 1} \right) + \nabla F\left(x^{k + 1}\right)^{T}y^{k + 1},
    		\end{align*}
    		where the second equality uses the the multiplier update given in \eqref{GenericAdap:MultiStep}.
    		Thus, using the latter, thanks to condition {\bf C2} we obtain for all $k \geq 0$ that there exists
    		$b > 0$ such that
    		\begin{equation}\label{lagb}
    			\norm{\nabla f_{0}\left(x^{k + 1}\right) + \nabla F\left(x^{k + 1}\right)^{T}y^{k + 1}} \leq
    			b\norm{x^{k + 1} - x^{k}}.
    		\end{equation}
    		Therefore, we obtain for all $k \geq \kinfo$,
    		\begin{align}
			\norm{\Delta_{k}} & = \norm{\nabla F\left(x^{k + 1}\right)^{T}y^{k + 1} - \nabla F\left(x^{k}
			\right)^{T}y^{k}} \nonumber \\
            	& = \norm{\nabla F\left(x^{k + 1}\right)^{T}y^{k + 1} + \nabla f_{0}\left(x^{k + 1}\right) -
            	\nabla F\left(x^{k}\right)^{T}y^{k} - \nabla f_{0}\left(x^{k}\right) + \nabla f_{0}\left(x^{k}
            	\right) - \nabla f_{0}\left(x^{k + 1}\right)} \nonumber \\
         	& \leq \norm{\nabla F\left(x^{k + 1}\right)^{T}y^{k + 1} + \nabla f_{0}\left(x^{k + 1}\right)} +
         	\norm{ \nabla F\left(x^{k}\right)^{T}y^{k} + \nabla f_{0}\left(x^{k}\right)} \nonumber \\
	 	 	& + \norm{\nabla f_{0}\left(x^{k + 1}\right) - \nabla f_{0}\left(x^{k}\right)} \nonumber \\
            	& \leq \left(L(f_{0}) + b\right)\norm{x^{k + 1} - x^{k}} + b\norm{x^{k} - x^{k - 1}},
            	\label{L:DualSequence:4}
       	\end{align}
        where the last inequality uses \eqref{lagb}, and the Lipschitz continuity of $\nabla f_{0}$ over
        $\Zone$ (see Assumption \ref{AssumptionB}(iii)). Combining \eqref{L:DualSequence:2} and
        \eqref{L:DualSequence:4}, we thus obtain for any $k \geq \kinfo$
        \begin{equation} \label{L:DualSequence:5}
	        	\gamma\norm{y^{k + 1} - y^{k}} \leq \left(L(f_{0}) + L(F)\Lambda + b\right)\norm{x^{k + 1} -
	        	x^{k}} + b\norm{x^{k} - x^{k - 1}}.
        \end{equation}
        Therefore, squaring the last inequality and using the fact that $\left(r + s\right)^{2} \leq 2r^{2}
        + 2s^{2}$ for all $r , s \in \rr$, the claimed assertion follows.
    \end{proof}

\subsection{Finite stabilization of the penalty sequence $\seq{\rho}{k}$.}
	We are now ready to establish the promised key lemma which asserts that the sequence of penalizing
	parameters $\seq{\rho}{k}$ becomes stationary from a certain iteration-index $\ksta$. A ``Lyapunov
	zone" for $\EE_{\beta}$ is thus reached within finitely many steps.
	\begin{lemma}[Finite stabilization of the sequence $\seq{\rho}{k}$] \label{L:SufficentDecrease}
		Let $\Seq{z}{k}$ be a Lagrangian sequence. Assume that the multiplier sequence $\Seq{y}{k}$ is
		bounded. Then, there exists an index $\ksta \in \nn$ such that
		\begin{equation*}
			\rho_{k} = \rho_{\ksta}, \quad \forall \,\, k \geq \ksta.
		\end{equation*}
		Moreover, for all $k \geq \ksta$ we have $x^{k} \in \Zone$, and there exists $\tau > 0$ such that
		\begin{equation} \label{L:SufficentDecrease:0}
 			\tau\norm{x^{k + 1} - x^{k}}^{2} \leq \EE_{\beta_{\ksta}} \left(x^{k} , u^{k} , y^{k} ,
 			x^{k -1}\right) - \EE_{\beta_{\ksta}}\left(x^{k + 1} , u^{k + 1} , y^{k + 1} , x^{k}
 			\right).
		\end{equation}
	\end{lemma}	
	\begin{proof}
		Lemma \ref{L:Stab} warrants that $x^{k} \in \Zone$ for all $k \geq \kinfo$ and by applying Lemma
		\ref{L:DescentProperty},  we obtain for all $k \geq 0$ that
		\begin{equation} \label{L:SufficentDecrease:1}
			\Laug_{\rho_{k}}\left(z^{k}\right) - \Laug_{\rho_{k}}\left(z^{k + 1}\right) \geq \frac{a}{2}
			\norm{x^{k + 1} - x^{k}}^{2} - \frac{1}{\rho_{k}}\norm{y^{k + 1} - y^{k}}^{2}.
		\end{equation}
		Using Lemma \ref{L:DualSequence}, we get for all $k \geq \kinfo$,
		\begin{equation} \label{L:SufficentDecrease:2}
			\norm{y^{k + 1} - y^{k}}^{2} \leq d_{1}\norm{x^{k + 1} - x^{k}}^{2} + d_{2}\norm{x^{k} - x^{k -
			1}}^{2},
		\end{equation}
		where $d_{1}$ and $d_{2}$ are given in \eqref{L:DualSequence:0}. Hence, by combining
		\eqref{L:SufficentDecrease:1} and \eqref{L:SufficentDecrease:2}, it follows for all $k \geq \kinfo$,
		that
		\begin{equation} \label{L:SufficentDecrease:3}
			\Laug_{\rho_{k}}\left(z^{k}\right) - \Laug_{\rho_{k}}\left(z^{k + 1}\right) \geq \left(\frac{a}
			{2} - \frac{d_{1}}{\rho_{k}}\right)\norm{x^{k + 1} - x^{k}}^{2} - \frac{d_{2}}{\rho_{k}}
			\norm{x^{k} - x^{k - 1}}^{2}.
		\end{equation}
		Using the definition of $\EE_{\beta}$ (see \eqref{D:Lyap}) and setting $\beta := \beta_{k}$ for all
		$k \geq 0$, we get
		\begin{align}
			V_{k} & := \EE_{\beta_{k}}\left(x^{k} , u^{k} , y^{k} , x^{k - 1}\right) - \EE_{\beta_{k}}
			\left(x^{k + 1} , u^{k + 1} , y^{k + 1} , x^{k}\right) \nonumber \\
    			& = \Laug_{\rho_{k}}\left(z^{k}\right) - \Laug_{\rho_{k}}\left(z^{k + 1}\right) + \beta_{k}
    			\norm{x^{k} - x^{k - 1}}^{2} - \beta_{k}\norm{x^{k + 1} - x^{k}}^{2}.
		\end{align}
		Therefore, with \eqref{L:SufficentDecrease:3}, we deduce that for all $k \geq \kinfo$
		\begin{align}
        		V_{k} & \geq \left(\frac{a}{2} - \frac{d_{1}}{\rho_{k}} - \beta_{k}\right)\norm{x^{k + 1} -
        		x^{k}}^{2} - \left(\frac{d_{2}}{\rho_{k}}  - \beta_{k}\right)\norm{x^{k - 1} - x^{k}}^{2}
        		\nonumber \\
        		& = \left(\frac{a}{2} - \frac{d_{1}}{\rho_{k}} - \beta_{k}\right)\norm{x^{k + 1} - x^{k}}^{2},
        		\label{L:SufficentDecrease:4}
		\end{align}
		where the equality follows from the definition of $\beta_{k}$ given in Step 3.3 of {\bf ALBUM}.
		Hence, using \eqref{L:DualSequence}, we get that
		\begin{equation*}
			\beta_{k} = \frac{d_{2}}{\rho_{k}}= \frac{2b^{2}}{\rho_{k}\gamma^{2}}.
		\end{equation*}
	 	In addition, one has for all $k \geq \kinfo$ that
		\begin{equation}
			\frac{a}{2} - \frac{d_{1}}{\rho_{k}} - \beta_{k} = \frac{a}{2} - \frac{d_{1} + d_{2}}{\rho_{k}}.
		\end{equation}
		Thus \eqref{L:SufficentDecrease:4} rewrites
		\begin{equation} \label{L:SufficentDecrease:5}
        		V_{k} \geq \left(\frac{a}{2} - \frac{d_{1} + d_{2}}{\rho_{k}}\right)\norm{x^{k + 1} - x^{k}}
        		^{2}.
		\end{equation}		
		The sequence $\seq{\rho}{k}$ cannot increase indefinitely else we would get from
		\eqref{L:SufficentDecrease:5} that
		\begin{equation*}
			\EE_{\beta_{k}}\left(x^{k} , u^{k} , y^{k} , x^{k -1}\right) - \EE_{\beta_{k}}\left(x^{k +
			1} , u^{k + 1} , y^{k + 1} , x^{k}\right) \geq \tau \norm{x^{k + 1} - x^{k}}^{2},
		\end{equation*}
		for all $k$ sufficiently large, where $\tau > 0$ is the parameter given in the {\bf ALBUM} scheme.
		Thus we obtain the existence of an iteration-index $\ksta \geq \kinfo$ such that $\rho_{k} =
		\rho_{\ksta}$ for all $k \geq \ksta$, and the desired result follows.
	\end{proof}
	\begin{remark}[Adaptive process and the dynamics of $\seq{\rho}{k}$] \label{liap}
		Lemma \ref{L:SufficentDecrease} establishes that {\bf ALBUM}, within Step 3.3, relies on two
		fundamental tests:
		\begin{itemize}
			\item[--] a weak\footnote{Weak because we do not ask for actual feasibility.} feasibility test,
				\ie $x^{k} \in \Zone$,
			\item[--] a surrogate\footnote{Surrogate because we do not ask for the augmented Lagrangian 	
				function $\Laug_{\rho}$ to be Lyapunov, but rather that the auxiliary function $\EE_{\beta}$
				is Lyapunov.} descent test for $\EE_{\beta}$ which implicitly tunes the algorithm to match
				the natural step-sizes attached to $f_{0}$ and $F$.
		\end{itemize}
		Lemma \ref{L:SufficentDecrease} tells us that $\rho_{k}$ can be automatically tuned to an acceptable
		value $\rho_{\ksta}$ in finitely many steps. As a consequence, and it is a fundamental fact, we have
		the descent property:		
		\begin{equation*}
			\EE_{\beta_{\ksta}}\left(x^{k} , u^{k} , y^{k} , x^{k -1}\right) - \EE_{\beta_{\ksta}}\left(x^{k
			+ 1} , u^{k + 1} , y^{k + 1} , x^{k}\right) \geq \tau \norm{x^{k + 1} - x^{k}}^{2}, \quad
			\forall \,\, k \geq \ksta.
		\end{equation*}	
		In short and to conclude, one could say that the adaptive protocol leads to the finite
		identification of the information zone and to a sufficient descent property.
	\end{remark}
	\begin{remark}
		One observes from the proof, that the descent property on $\EE_{\beta}$ is ensured once we know that
		\begin{equation} \label{descond}
			\frac{a}{2} - \frac{d_{1} + d_{2}}{\rho_{k}} > \tau, \quad \forall \,\, k \geq \kinfo.
		\end{equation}
		In order to shunt the surrogate descent test, it is thus tempting to fix a value $\rho_{0}$ a priori
		(before running the method), so that the above holds directly. Yet it is important to understand
		that this cannot be done in general, since $d_{1}$ (\cf \eqref{L:DualSequence:0}) is a constant that
		{\em depends} on a bound $\Lambda$ of the sequence $\Seq{y}{k}$ which by itself depends on
		$\seq{\rho}{k}$!
	\end{remark}
\medskip

	\begin{remark}[Special case with $F$ assumed to be linear] \label{r:lin}
		\begin{itemize}
			\item[(i)] In that case the dependence of $d_{1}$ with $\Lambda$ given in Lemma
				\ref{L:DualSequence} {\em disappears}. This allows for a more direct and simplified
				approach. Indeed, exploiting the linearity of $F$, the inequality \eqref{L:DualSequence:2}
				reduces to $\norm{\Delta_{k}} \geq \gamma\norm{y^{k + 1} - y^{k}}$ for all $k \geq 0$, where
				here $\gamma \equiv \sqrt{\lambda_{\min}(FF^T)} > 0$, \cf Remark \ref{r:info}. Therefore,
				the boundedness of $\Seq{y}{k}$ is not needed, and it immediately follows that the proof of
				inequality \eqref{L:DualSequence:00} holds true in Lemma \ref{L:DualSequence} for all $k
				\geq 0$, with
				\begin{equation} \label{dlin}
					d_{1} = \frac{2}{\lambda_{\min}(FF^T)}\left(L(f_{0})  + b\right)^{2} \quad \text{and}
					\quad d_{2} = \frac{2b^{2}}{\lambda_{\min}(FF^T)}.
				\end{equation}
				Secondly, as mentioned before (\cf Remark \ref{r:info}) the information zone can be taken as
				the whole space \ie $\Zone \equiv \rr^{n}$, and in that case the adaptive regime is not
				anymore necessary. Thus we set $\rho_{k} \equiv \rho > 0$ for all $k \in \nn$, and Step 3.3
				of {\bf ALBUM} is simply removed (see also Remark \ref{r:param2}). Therefore, in order to
				guarantee sufficient descent of the Lyapunov $\EE_{\beta}$, all we need is that
				\eqref{descond} holds true, that is (with $\tau = 0$), it reduces to
    				\begin{equation} \label{rholin}
					\rho > {\bar \rho} := \frac{2\left(d_{1} + d_{2}\right)}{a},
				\end{equation}
		 		where $d_{1}$ and $d_{2}$ are given in \eqref{dlin}. Therefore, in the special linear case,
		 		this allows for determining explicitly the threshold value ${\bar \rho}$, for a chosen
		 		Lagrangian algorithmic map $ {\cal A_\rho}$ which provides the constants $a$ and $b$ and to 
		 		obtain the corresponding convergence results via a straightforward application of Theorems 
		 		\ref{T:SubConv} and \ref{T:GlobConv}.
			\item[(ii)] Interestingly, this also provides a positive answer to a question posed in
				\cite[Remark 4(3) p. 2451]{LP2015}, where the authors pointed out that it would be
				interesting to see if global convergence of a proximal ADM could be derived; see also 
				Section \ref{sec:variants} for more results.
		\end{itemize}
	\end{remark}

\section{Proof of the main convergence results.} \label{Sec:proofs}
	Equipped with the results we have established, we can now apply our methodology to prove the main
	convergence results of {\bf ALBUM} announced in Section \ref{mainconv}.
	
\subsection{Subgradient bound for the Lyapunov function $\EE_{\beta}$.}
	As mentioned previously, we work with the function $\EE_{\beta}$ to overcome the descent obstacle and to 
	detect hidden descent mechanisms. Now the third condition {\bf C3} of our methodology comes into a play. 
	We derive below an upper bound on a subgradient of the Lyapunov function $\EE_{\beta}$.
    	\begin{lemma} \label{L:SubgradientBound}
		Let $\Seq{z}{k}$ be a bounded Lagrangian sequence. Then, for each $k \in \nn$, there exist positive
		constants $\s_{1}$ and $\s_{2}$ together with $q^{k + 1} \in \partial \EE_{\beta_{k}}\left(x^{k + 1}
		, u^{k + 1} , y^{k + 1} , x^{k}\right)$, such that for all $k \geq \kinfo$
		\begin{equation} \label{L:SubgradientBound:00}
			\norm{q^{k + 1}} \leq \s_{1}\norm{x^{k + 1} - x^{k}} + \s_{2}\norm{x^{k} - x^{k - 1}}.
		\end{equation}
    	\end{lemma}	
    	\begin{proof}
    		Consider the quadruplet $q^{k + 1} = \left(q_{1}^{k + 1} , q_{2}^{k + 1} , q_{3}^{k + 1} , q_{4}^{k
    		+ 1}\right) \in \partial \EE_{\beta_{k}}\left(x^{k + 1} , u^{k + 1} , y^{k + 1} , x^{k}\right)$.
    		Using the definition of $\EE_{\beta}$ (see \eqref{D:Lyap}), subdifferential calculus rules, and
    		recalling the multiplier update rule \eqref{GenericAdap:MultiStep}, a direct computation shows that:
	    	\begin{align}
    			q_{1}^{k + 1} & = \nabla_{x} \Laug_{\rho_{k}}\left(x^{k + 1} , u^{k + 1} , y^{k + 1}\right) +
    			2\beta_{k}\left(x^{k + 1} - x^{k}\right) \nonumber \\
    			& = \nabla_{x} \Laug_{\rho_{k}}\left(x^{k + 1} , u^{k + 1} , y^{k}\right) + \nabla F\left(x^{k +
    			1}\right)^{T}\left(y^{k + 1} - y^{k}\right) + 2\beta_{k}\left(x^{k + 1} - x^{k}\right),
    			\label{L:SubgradientBound:1} \\
    			q_{2}^{k + 1} & \in \partial_{u} \Laug_{\rho_{k}}\left(x^{k + 1} , u^{k + 1} , y^{k + 1}\right)
    			= \partial_{u} \Laug_{\rho_{k}}\left(x^{k + 1} , u^{k + 1} , y^{k}\right) - \left(y^{k + 1} -
    			y^{k}\right), \label{L:SubgradientBound:2} \\
    			q_{3}^{k + 1} & = \nabla_{y} \Laug_{\rho_{k}}\left(x^{k + 1} , u^{k + 1} , y^{k + 1}\right) =
    			F\left(x^{k + 1}\right) - u^{k + 1} = \rho_{k}^{-1}\left(y^{k + 1} - y^{k}\right),
    			\label{L:SubgradientBound:3} \\
    			q_{4}^{k + 1} & = 2\beta_{k}\left(x^{k} - x^{k + 1}\right). \label{L:SubgradientBound:4}
    		\end{align}   		
       	Since $\Seq{x}{k}$ is assumed bounded and $\nabla F$ is continuous (see Assumption
       	\ref{AssumptionB}(ii)) it follows that there exists $B > 0$ such that
     	\begin{equation} \label{L:SubgradientBound:0}
     		\sup_{k \geq \kinfo} \norm{\nabla F\left(x^{k}\right)} \leq B.
     	\end{equation}
     	Moreover, recall that from \eqref{parcb}, we have $\rho_{k} \geq \rho_{0}$ and $\beta_{k} \leq
     	\beta_{0}$ for all $k \in \nn$. Therefore, using condition {\bf C2} and the expressions for
     	$q_{j}^{k + 1}$, $j = 1 , 2 , 3 , 4$ derived above, we get the following estimates:
     	\begin{align*}
     		\norm {q_{1}^{k + 1}} & \leq \norm {\nabla_{x} \Laug_{\rho_{k}}\left(x^{k + 1} , u^{k + 1} ,
     		y^{k}\right)} + B\norm{y^{k + 1} - y^{k}} + 2\beta_{0}\norm {x^{k + 1} - x^{k}} \\
     		& \leq b\norm{x^{k + 1} - x^{k}} +  B\norm{y^{k + 1} - y^{k}} + 2\beta_{0}\norm {x^{k + 1} -
     		x^{k}} \\
     		& = B\norm{y^{k + 1} - y^{k}} + \left(b + 2\beta_{0}\right)\norm {x^{k + 1} - x^{k}}.
   	  	\end{align*}
     	Likewise, thanks to condition {\bf C3} we have with $v^{k + 1} \in \partial_{u} \Laug_{\rho_{k}}
     	\left(x^{k + 1} , u^{k + 1} , y^{k}\right)$ that $\norm{v^{k + 1}} \leq d\norm{x^{k + 1} - x^{k}}$,
     	and hence by defining $q_{2}^{k + 1} = v^{k + 1} - \left(y^{k + 1} - y^{k}\right)$, it immediately
     	follows that $q_{2}^{k + 1} \in \partial_{u} \Laug_{\rho_{k}}\left(x^{k + 1} , u^{k + 1} , y^{k + 1}
     	\right)$, and from \eqref{L:SubgradientBound:2}
		\begin{equation*}
			\norm{q_{2}^{k + 1}} \leq \norm{v^{k + 1}} + \norm{y^{k + 1} - y^{k}} \leq d\norm{x^{k + 1} -
			x^{k}} + \norm{y^{k + 1} - y^{k}}.
		\end{equation*}	
    		Finally, from \eqref{L:SubgradientBound:3} and \eqref{L:SubgradientBound:4} we immediately obtain
    		(recall \eqref{parcb})
    		\begin{equation*}
    			\norm{q_{3}^{k + 1}} \leq \frac{1}{\rho_{0}}\norm{y^{k + 1} - y^{k}} \quad \text{and} \quad
    			\norm{q_{4}^{k + 1}} \leq 2\beta_{0}\norm{x^{k + 1} - x^{k}}.
    		\end{equation*}
     	Therefore, summing these inequalities, we obtain for all $k \geq \ksta$
     	\begin{equation*}
    			\norm{q^{k + 1}} \leq \sum_{j = 1}^{4} \norm{q_{j}^{k + 1}} \leq \left(B + 1 + \frac{1}    	
    			{\rho_{0}}\right)\norm{y^{k + 1} - y^{k}} + \left(4\beta_{0} + b + d\right)\norm{x^{k + 1} -
    			x^{k}}.
    		\end{equation*}
    		Using the proof of Lemma \ref{L:DualSequence}, for all $k \geq \ksta$, we know from
    		\eqref{L:DualSequence:5} that
    		\begin{equation}
	        	\gamma\norm{y^{k + 1} - y^{k}} \leq \left(L(f_{0}) + L(F)\Lambda + b\right)\norm{x^{k + 1} -
	        	x^{k}} + b\norm{x^{k} - x^{k - 1}}.
        \end{equation}
        Combining this with the above inequality yields the desired estimation \eqref{L:SubgradientBound:00}
        by choosing
    		\begin{equation*}
    			\s_{1} = \frac{1}{\gamma}\left(B + 1 + \frac{1}{\rho_{0}}\right)\left(L(f_{0}) + L(F)\Lambda +
    			b\right) + 4\beta_{0} + b + d \quad \text{and} \quad \s_{2} = \frac{b}{\gamma}\left(B + 1 +
    			\frac{1}{\rho_{0}}\right).
    		\end{equation*}    		
    		This completes the proof.
    \end{proof}
	Equipped with Lemma \ref{L:SufficentDecrease} we immediately obtain the following result.
	\begin{proposition} \label{P:WeakConv}
		Let $\Seq{z}{k}$ be a Lagrangian sequence. Assume that the multiplier sequence $\Seq{y}{k}$ is
		bounded. Then
		\begin{equation*}
			\sum_{k = 1}^{\infty} \norm{x^{k + 1} - x^{k}}^{2} < \infty \quad \text{and} \quad \sum_{k = 1}
			^{\infty} \norm{y^{k + 1} - y^{k}}^{2} < \infty.
		\end{equation*}
	\end{proposition}
	\begin{proof}
		Invoking Lemma \ref{L:SufficentDecrease} which holds true under the stated assumptions, we have that
		\begin{equation} \label{P:WeakConv:1}
 			\tau\norm{x^{k + 1} - x^{k}}^{2} \leq \EE_{\beta_{\ksta}} \left(x^{k} , u^{k} , y^{k} , x^{k -
 			1}\right) - \EE_{\beta_{\ksta}}\left(x^{k + 1} , u^{k + 1} , y^{k + 1} , x^{k}\right),
		\end{equation}
		for all $k \geq \ksta$. Summing \eqref{P:WeakConv:1} over $k = \ksta , \ksta + 1 , \ldots , \ksta +
		p$ we obtain
		\begin{align*}
    	    		\tau\sum_{k = \ksta}^{\ksta + p} \norm{x^{k + 1} - x^{k}}^{2} & \leq \EE_{\beta_{\ksta}}
    	    		\left(x^{1} , u^{1} , y^{1} , x^{0}\right) - \EE_{\beta_{\ksta}}\left(x^{p + 1} , u^{p + 1} ,
    	    		y^{p + 1} , x^{p}\right) \\
   	 		& \leq \EE_{\beta_{\ksta}}\left(x^{1} , u^{1} , y^{1} , x^{0}\right),
		\end{align*}
		where the last inequality follows from the fact that $\inf_{(x , u)} \EE_{\beta_{\ksta}} > - \infty$
		(thanks to \eqref{WellPosed} since $\EE_{\beta_{\ksta}}\left(\cdot\right) \geq \Laug_{\rho_{\ksta}}
		\left(\cdot\right)$). Letting $p \rightarrow \infty$ yields
		\begin{equation*}
			\sum_{k = 1}^{\infty} \norm{x^{k + 1} - x^{k}}^{2} < \infty.
		\end{equation*}
		Therefore, from Lemma \ref{L:DualSequence}, it also follows that $\sum_{k = 1}^{\infty} \norm{y^{k +
		1} - y^{k}}^{2} < \infty$, as required.
	\end{proof}
	We are now ready to prove our first convergence result for the generic scheme {\bf ALBUM}.

\subsection{Proof of Theorem \ref{T:SubConv} -- subsequence convergence.}
	The sequence $\Seq{z}{k}$ is bounded and therefore there exists a subsequence $\left\{ z^{m_{k}}
	\right\}_{k \in \nn}$ which converges to ${\bar z} = \left({\bar x} , {\bar u} , {\bar y}\right)$. We
	first prove that $\left({\bar x} , {\bar u} , {\bar y} , {\bar x}\right)$ is a critical point of
	$\EE_{\beta_{\ksta}}$, that is,
	\begin{equation*}
		\left(0 , 0 , 0 , 0\right) \in \partial \EE_{\beta_{\ksta}}\left({\bar x} , {\bar u} , {\bar y} ,
		{\bar x}\right).
	\end{equation*}
	Since $h$ is lower semi-continuous we have
	that
    \begin{equation*}
		\liminf_{k \rightarrow \infty} h\left(u^{m_{k}}\right) \geq h\left({\bar u}\right),
   	\end{equation*}
  	which combined with condition {\bf C4} yields that $h\left(u^{m_{k}}\right)$ converges to $h\left({\bar
  	u}\right)$ as $k \rightarrow \infty$. Therefore, from Proposition \ref{P:WeakConv} and the continuity of
  	$f_{0}$ and $F$ (see Assumption \ref{AssumptionB}(ii) and (iii)), we obtain that
   	\begin{align*}
		\lim_{k \rightarrow \infty} \EE_{\beta_{\ksta}}\left(z^{m_{k}} , x^{m_{k} - 1}\right) & = \lim_{k
		\rightarrow \infty} \left[\Laug_{\rho_{\ksta}}\left(x^{m_{k}} , u^{m_{k}} , y^{m_{k}}\right) +
		\beta_{\ksta}\norm{x^{m_{k}} - x^{m_{k} - 1}}^{2}\right] \\
       	& = \Laug_{\rho_{\ksta}}\left({\bar x} , {\bar u} , {\bar y}\right) \\
        & = \EE_{\beta_{\ksta}}\left({\bar z} , {\bar x}\right).
	\end{align*}
   	We know from Lemma \ref{L:SubgradientBound} that there exist $\s_{1} , \s_{2} > 0$ and $q^{k + 1} \in
   	\partial \EE_{\beta_{\ksta}}\left(z^{k + 1} , x^{k}\right)$ for which
	\begin{equation*}
		\norm{q^{k + 1}} \leq \s_{1}\norm{x^{k + 1} - x^{k}} + \s_{2}\norm{x^{k} - x^{k - 1}}.
	\end{equation*}
	On the other hand, from Proposition \ref{P:WeakConv} it follows that
    	\begin{equation*}
		\lim_{k \rightarrow \infty} \norm{x^{k + 1} - x^{k}} = 0.
	\end{equation*}
    	Thus $q^{k + 1} \rightarrow 0$ as $k \rightarrow \infty$. Using the closedness property of the graph
    	of the subdifferential $\partial \EE_{\beta}$, we obtain that $\left(0 , 0 , 0 , 0\right) \in \partial
    	\EE_{\beta_{\ksta}}\left({\bar x} , {\bar u} , {\bar y} , {\bar x}\right)$. This shows that $\left({\bar
    	x} , {\bar u} , {\bar y} , {\bar x}\right)$ is a critical point of $\EE_{\beta_{\ksta}}$. Proposition
    	\ref{P:Crit} now implies that ${\bar x}$ is a critical point of the objective function $f$ of model
    	(CM), and the proof is completed.
\medskip

	Next, in order to prove the second main global convergence result of our algorithm {\bf ALBUM}, we need
	to introduce adequate and necessary material on the nonsmooth KL property \cite{BDL2006}.
\medskip

    Let $\eta \in \left(0 , +\infty\right]$. We denote by $\Phi_{\eta}$ the class of all concave and
    continuous functions $\varphi : \left[0 , \eta\right) \rightarrow \rr_{+}$ which satisfy the following
    conditions
    \begin{itemize}
        \item[$\rm{(i)}$] $\varphi\left(0\right) = 0$;
        \item[$\rm{(ii)}$] $\varphi$ is $C^{1}$ on $\left(0 , \eta\right)$ and continuous at $0$;
        \item[$\rm{(iii)}$] for all $s \in \left(0 , \eta\right)$: $\varphi'\left(s\right) > 0$.
    \end{itemize}
    The next result plays a crucial role, see \cite[Lemma 6]{BST2014}.
 	\begin{lemma}[Uniformized KL property] \label{L:KLProperty}
        Let $\Omega$ be a compact set and let $\s : \rr^{d} \rightarrow \left(-\infty , \infty\right]$ be a
        proper and lower semicontinuous function. Assume that $\s$ is constant on $\Omega$ and satisfies the
        KL property at each point of $\Omega$. Then, there exist $\varepsilon > 0$, $\eta > 0$ and $\varphi
        \in \Phi_{\eta}$ such that for all $\overline{u}$ in $\Omega$ and all $u$ in the following
        intersection
        \begin{equation} \label{L:KLProperty:1}
            \left\{ u \in \rr^{d} : \, \dist\left(u , \Omega\right) < \varepsilon \right\} \cap \left[
            \s\left(\overline{u}\right) < \s\left(u\right) < \s\left(\overline{u}\right) + \eta\right],
        \end{equation}
        one has,
        \begin{equation} \label{L:KLProperty:2}
            \varphi'\left(\s\left(u\right) - \s\left(\overline{u}\right)\right)\dist\left(0 , \partial
            \s\left(u\right)\right) \geq 1.
        \end{equation}
    \end{lemma}
 	Equipped with these results we proceed with the proof of the second main theorem, \ie convergence of the
 	whole sequence $\Seq{z}{k}$ to a critical point of problem (CM) with semi-algebraic data $f_{0}$, $h$
 	and $F$. Note that the technique used below is patterned after the recent work \cite{BST2014}. However,
 	as explained previously, we cannot apply directly these results to {\bf ALBUM}, since the descent
 	requirements stated there clearly do not hold in our framework.
		
\subsection{Proof of Theorem \ref{T:GlobConv} -- global convergence.}
	Since $\Seq{z}{k}$ is bounded there exists a subsequence $\left\{ z^{m_{k}} \right\}_{k \in \nn}$ such
	that $z^{m_{k}} \rightarrow {\bar z}$ as $k \rightarrow \infty$. In a similar way as in Theorem
	\ref{T:SubConv} we get that
   	\begin{equation} \label{T:FiniteLength:1}
		\lim_{k \rightarrow \infty} \EE_{\beta_{\ksta}}\left(z^{k} , x^{k - 1}\right) = \EE_{\beta_{\ksta}}
		\left({\bar z} , {\bar x}\right).
	\end{equation}
	If there exists an integer $\bar{k} \geq \ksta$ for which $\EE_{\beta_{\ksta}}\left(z^{\bar{k}} ,
	x^{\bar{k} - 1}\right) = \EE_{\beta_{\ksta}}\left({\bar z} , {\bar x}\right)$ then the decreasing
	property obtained in Lemma \ref{L:SufficentDecrease} would imply that $z^{\bar{k} + 1} = z^{\bar{k}}$. A
	trivial induction show then that the sequence $\Seq{z}{k}$ is stationary and the announced result is
	obvious.
\medskip

	Since $\left\{ \EE_{\beta_{\ksta}}\left(z^{k} , x^{k - 1}\right) \right\}_{k \in \nn}$ is a
	nonincreasing sequence, it is clear from \eqref{T:FiniteLength:1} that
	\begin{equation*}
		\EE_{\beta_{\ksta}}\left({\bar z} , {\bar x}\right) < \EE_{\beta_{\ksta}}\left(z^{k} , x^{k - 1}
		\right)\; \text{for all} \; k \geq \ksta.
	\end{equation*}
	Again from \eqref{T:FiniteLength:1}, for any $\eta > 0$ there exists $k_{0} \geq \ksta$ such that
   	\begin{equation*}
		\EE_{\beta_{\ksta}}\left(z^{k} , x^{k - 1}\right) < \EE_{\beta_{\ksta}}\left({\bar z} , {\bar x}
		\right) + \eta, \, \forall \, k > k_{0}.
	\end{equation*}
	From Theorem \ref{T:SubConv} we know that $\lim_{k \rightarrow \infty} \dist\left(z^{k} , \omega
	\left(z^{0}\right)\right) = 0$. This means that for any $\varepsilon > 0$ there exists a positive
	integer $k_{1} \geq \ksta$ such that $\dist\left(z^{k} , \omega\left(z^{0}\right)\right) < \varepsilon$
	for all $k > k_{1}$. Summing up all these facts, we get that $z^{k}$ belongs to the intersection in
	\eqref{L:KLProperty:1} for all $k > l := \max\left\{ k_{0} , k_{1} \right\} \geq \ksta$.
\medskip

	We denote by $\omega\left(z^{0}\right)$ the set of all limit points. By Theorem \ref{T:SubConv}, $\omega
	\left(z^{0}\right)$ is nonempty and  compact (since by definition, it can viewed as an intersection of
	compact sets). Now, we show that  $\EE_{\beta_{\ksta}}$ is finite and constant on $\omega\left(z^{0}
	\right)$. Indeed, by our standing assumption (see \eqref{WellPosed}) we know that $\Laug_{\rho}
	\left(z^{k} \right) > -\infty$ for all $k \in \nn$, therefore from the definitions of $\Laug_{\rho}$ and
	$\EE_{\beta}$ (see \eqref{D:AugLAc} and \eqref{D:Lyap}, respectively) we have that $\left\{
	\EE_{\beta_{k}}\left(z^{k} , x^{k - 1}\right) \right\}_{k \in \nn}$ is bounded from below. Lemma
	\ref{L:SufficentDecrease} now guarantees that $\left\{ \EE_{\beta_{\ksta}}\left(z^{k} , x^{k - 1}\right)
	\right\}_{k \in \nn}$ converges to a finite limit, say $l$. From \eqref{T:FiniteLength:1} it follows
	that $l = \EE_{\beta_{\ksta}}\left({\bar z} , {\bar x}\right)$, which proves that $\EE_{\beta_{\ksta}}$
	is finite and constant on $\omega\left(z^{0}\right)$.
\medskip

	Thus, since $\EE_{\beta_{\ksta}}$ is a KL function, we can  apply the Uniformization Lemma
	\ref{L:KLProperty} with $\Omega = \omega\left(z^{0}\right)$. Therefore, for any $k > l$, we have
   	\begin{equation} \label{T:FiniteLength:2}
		\varphi'\left(\EE_{\beta_{\ksta}}\left(z^{k} , x^{k - 1}\right) - \EE_{\beta_{\ksta}}\left({\bar z}
		, {\bar x}\right)\right) \cdot \dist\left(0 , \partial \EE_{\beta_{\ksta}}\left(z^{k} , x^{k - 1}
		\right)\right) \geq 1.
	\end{equation}
   	This makes sense since we know that $\EE_{\beta_{\ksta}}\left(z^{k} , x^{k - 1}\right) >
   	\EE_{\beta_{\ksta}}\left({\bar z} , {\bar x}\right)$ for any $k > l \geq \ksta$. Using Lemma
	\ref{L:SubgradientBound} (recalling that $\ksta \geq \kinfo$), we get that
	\begin{align}
		\varphi'\left(\EE_{\beta_{\ksta}}\left(z^{k} , x^{k - 1}\right) - \EE_{\beta_{\ksta}}\left({\bar z}
		, {\bar x}\right)\right) & \geq \frac{1}{\dist\left(0 , \partial \EE_{\beta_{\ksta}}\left(z^{k} ,
		x^{k - 1}\right)\right)} \nonumber \\
      	& \geq \left(\s\norm{x^{k} - x^{k - 1}} + \s\norm{x^{k - 1} - x^{k - 2}}\right)^{-1},
      	\label{T:FiniteLength:3}
	\end{align}
	where $\s = \max \left\{ \s_{1} , \s_{2} \right\}$ while $\s_{1}$ and $\s_{2}$ given in Lemma
	\ref{L:SubgradientBound}. On the other hand, from the concavity of $\varphi$ we get that
	\begin{align}
		\varphi\left(\EE_{\beta_{\ksta}}\left(z^{k} , x^{k - 1}\right) - \EE_{\beta_{\ksta}}\left({\bar z} ,
		{\bar x}\right)\right) - \varphi\left(\EE_{\beta_{\ksta}}\left(z^{k + 1} , x^{k}\right) -
		\EE_{\beta_{\ksta}}\left({\bar z} , {\bar x}\right)\right) \geq & \nonumber \\
		& \hspace{-5.5in} \varphi'\left(\EE_{\beta_{\ksta}}\left(z^{k} , x^{k - 1}\right) -
		\EE_{\beta_{\ksta}}\left({\bar z} , {\bar x}\right)\right)\left(\EE_{\beta_{\ksta}}\left(z^{k} ,
		x^{k - 1}\right) - \EE_{\beta_{\ksta}}\left(z^{k + 1} , x^{k}\right)\right).
		\label{T:FiniteLength:4}
	\end{align}
	For convenience, we define for all $p , q \in \nn$ and ${\bar z}$ the following quantities
	\begin{equation*}
		\Delta_{p , q} : = \varphi\left(\EE_{\beta_{\ksta}}\left(z^{p} , x^{p - 1}\right) -
		\EE_{\beta_{\ksta}}\left({\bar z} , {\bar x}\right)\right) - \varphi\left(\EE_{\beta_{\ksta}}
		\left(z^{q} , x^{q - 1}\right) - \EE_{\beta_{\ksta}}\left({\bar z} , {\bar x}\right)\right).
	\end{equation*}
	Combining \eqref{T:FiniteLength:3} and \eqref{T:FiniteLength:4} and using Lemma
	\ref{L:SufficentDecrease} yields for any $k > l$ that
	\begin{equation} \label{T:FiniteLength:5}
		\Delta_{k , k + 1} \geq \frac{\tau\norm{x^{k + 1} - x^{k}}^{2}}{\psi\left(\norm{x^{k} - x^{k - 1}} +
		\norm{x^{k - 1} - x^{k - 2}}\right)},
	\end{equation}
	and hence
	\begin{equation*}
		\norm{x^{k + 1} - x^{k}}^{2} \leq \rho\Delta_{k , k + 1}\left(\norm{x^{k} - x^{k - 1}} + \norm{x^{k
		- 1} - x^{k - 2}}\right),
	\end{equation*}
	where $\rho = \s/\tau$. Using the fact that $2\sqrt{\alpha\beta} \leq \alpha + \beta$ for all $\alpha ,
	\beta \geq 0$, we infer
	\begin{equation} \label{T:FiniteLength:6}
		4\norm{x^{k + 1} - x^{k}} \leq \norm{x^{k} - x^{k - 1}} + \norm{x^{k - 1} - x^{k - 2}} + 4\rho
		\Delta_{k , k + 1}.
	\end{equation}
	Let us now prove that for any $k > l$ the following inequality holds
	\begin{equation*}
		2\sum_{i = l + 1}^{k} \norm{x^{i + 1} - x^{i}} \leq 2\norm{x^{l + 1} - x^{l}} + \norm{x^{l} - x^{l -
		1}} + \rho\Delta_{l + 1 , k + 1}.
	\end{equation*}
	Summing up \eqref{T:FiniteLength:6} for $i = l + 1 , \ldots , k$ yields
	\begin{align*}
		4\sum_{i = l + 1}^{k} \norm{x^{i + 1} - x^{i}} & \leq \sum_{i = l + 1}^{k} \norm{x^{i} - x^{i - 1}}
		+ \sum_{i = l + 1}^{k} \norm{x^{i - 1} - x^{i - 2}} + 4\rho\sum_{i = l + 1}^{k} \Delta_{i , i + 1}
		\\
		& \leq \sum_{i = l + 1}^{k} \norm{x^{i + 1} - x^{i}} + \norm{x^{l + 1} - x^{l}} + 4\rho\sum_{i = l +
		1}^{k} \Delta_{i , i + 1} \\
		& + \sum_{i = l + 1}^{k} \norm{x^{i + 1} - x^{i}} + \norm{x^{l + 1} - x^{l}} + \norm{x^{l} - x^{l -
		1}} \\
		& = 2\sum_{i = l + 1}^{k} \norm{x^{i + 1} - x^{i}} + 2\norm{x^{l + 1} - x^{l}} + \norm{x^{l} - x^{l
		- 1}} + 4\rho\Delta_{l + 1 , k + 1},
	\end{align*}
	where the last inequality follows from the fact that $\Delta_{p , q} + \Delta_{q , r} = \Delta_{p , r}$
	for all $p , q , r \in \nn$. Since $\varphi \geq 0$, we thus have for any $k > l$ that
	\begin{equation*}
		2\sum_{i = l + 1}^{k} \norm{x^{i + 1} - x^{i}} \leq 2\norm{x^{l + 1} - x^{l}} + \norm{x^{l} - x^{l -
		1}} + \gamma\varphi\left(\EE_{\beta_{\ksta}}\left(z^{l} , x^{l - 1}\right) - \EE_{\beta_{\ksta}}
		\left({\bar z} , {\bar x}\right)\right).
	\end{equation*}
	Since the right hand-side of the inequality above does not depend on $k$ at all, it is easily shows that
	the sequence $\Seq{x}{k}$ has finite length, that is,
	\begin{equation} \label{T:FiniteLength-Item1:6}
		\sum_{k = 1}^{\infty} \norm{x^{k + 1} - x^{k}} < \infty.
	\end{equation}
	This means that it is a Cauchy sequence and hence a convergent sequence. In addition, from
	\eqref{L:DualSequence:5} we also have
	\begin{equation*}
        	\sum_{k = 1}^{\infty} \norm{y^{k + 1} - y^{k}} < \infty,
	\end{equation*}
	and thus $\Seq{y}{k}$ has also finite length and therefore a convergent sequence. Now, the multiplier
	Step 3.2 yields, for any $k \geq \ksta$, that
	\begin{equation*}
		u^{k + 1} = F\left(x^{k + 1}\right) + \frac{1}{\rho_{\ksta}}\left(y^{k} - y^{k + 1}\right).
	\end{equation*}
	Since $F$ is continuous, $\Seq{x}{k}$ a convergent sequence and thanks to Proposition \ref{P:WeakConv}
	it follows that $\Seq{u}{k}$ is also a convergent sequence. From Theorem \ref{T:SubConv} it is clear
	that $\left\{\left(z^{k} , x^{k - 1}\right) \right\}_{k \in \nn}$ converges to a critical point
	$\left({\bar x} , {\bar u} , {\bar y} , {\bar x}\right)$ of $\EE_{\beta_{\ksta}}$. We finally conclude
	from Proposition \ref{P:Crit} that ${\bar x}$ is a critical point of $f$.

\section{Applications: specific schemes from {\bf ALBUM}.} \label{sec:variants}
	The generic scheme {\bf ALBUM} encompasses interesting Lagrangian based methods. First recall that in
	any Lagrangian based method, the multiplier update is always given by an {\em explicit} formula (see
	\eqref{GenericAdap:MultiStep}):
	\begin{equation*}
		y^{k + 1} = y^{k} + \rho_{k}\left(F\left(x^{k + 1}\right) - u^{k + 1}\right).
  	\end{equation*}
	Thus, the main computational and algorithmic issues which emerge from {\bf ALBUM} depend on the way we
	define the Lagrangian algorithmic map $\AAA_{\rho}$ to compute the primal step. In general, any
	minimization algorithm can be used at this stage. We focus on the description of two fundamental types
	of maps $\AAA_{\rho}$, yet we note that other variants can also be conceived depending on the problem's
	data information and the structure at hand. This point will be  further developed below in Section
	\ref{other}.

\subsection{Two fundamental instances of $\AAA_{\rho}$ and the corresponding {\bf ALBUM}.}
	Given a triple $\left(x^{k} , u^{k} , y^{k}\right)$ we compute the next primal variables $x^{k + 1}$ and
	$u^{k + 1}$ in {\bf ALBUM} via the algorithmic map $\AAA_{\rho}$ given by either one of the following
	minimization schemes:
	\begin{itemize}
		\item {\bf ALBUM 1}  -- {\em Joint Minimization $\equiv$ Proximal Multipliers Method \cite{R1976}}
			\begin{equation} \label{AAA-PMM}
				\left(x^{k + 1} , u^{k + 1}\right) \in \argmin_{(x , u)} \left\{ \Laug_{\rho_{k}}\left(x , u
				, y^{k}\right) + \frac{\mu}{2}\norm{x - x^{k}}^{2} \right\}, \quad (\mu > 0).
			\end{equation}
			This simple idea consists in minimizing a proximal counterpart of the augmented Lagrangian
			$\Laug_{\rho}$, jointly with respect to both primal variables $x$ and $u$, is nothing else but
			the classical dynamic of Proximal Method of Multipliers (PMM) of Rockafellar \cite{R1976}.
		\item {\bf ALBUM 2 } -- {\em Alternating Minimization (aka Gauss-Seidel) $\equiv$ Proximal ADM
			\cite{GLT1989}}
		
			Update the variables $x$ and $u$ in an alternating fashion as follows:
			\begin{align}
				u^{k + 1} & \in \argmin_{u} \Laug_{\rho_{k}}\left(x^{k} , u , y^{k}\right),
				\label{AAA-ADPMM:Step1} \\
				x^{k + 1} & \in \argmin_{x} \left\{ \Laug_{\rho_{k}}\left(x , u^{k + 1} , y^{k}\right) +
				\frac{\mu}{2}\norm{x - x^{k}}^{2} \right\}, \quad (\mu > 0). \label{AAA-ADPMM:Step2}
			\end{align}
	\end{itemize}
	\begin{remark} \label{albums}
		\begin{itemize}
			\item[(i)] Note that in the above two schemes the proximal regularization term was added only
				for the primal variable $x$ of the augmented Lagrangian, since by the construction of
				$\Laug_{\rho}$ (see \eqref{D:AugLAc}), we note that the primal variable $u$ already admit a
				built-in proximal term.
			\item[(ii)] Also, note that the flexibility of {\bf ALBUM} provides potential for further
				studies within other strategies or variants that could be conceived and further developed in
				future work, \eg adding a proximal regularization term for $u$ around $u^{k}$ and performing
				a subgradient step for determining the next point $u^{k + 1}$; or dropping one of the
				proximal regularization term in exchange of other assumptions on the problem's data, see 
				section \ref{other} for the latter  situation.
		\end{itemize}
	\end{remark}
\medskip

	\begin{remark}[Tractability of the subproblems] \label{tract}
		Although the practical aspects involving implementation are beyond the scope of this work, it is 
		important to discuss  some of these issues. In this regard we comment the general practicability of 
		the steps of {\bf ALBUM 2} whose alternating structure is often more favorable toward 
		implementation. Recall that {\bf ALBUM 2} features a simple dual step and two primal steps \`a la 
		Gauss-Seidel, one with respect to $u$ and one with respect to $x$, we discuss them below:
\medskip

		\begin{itemize}
			\item[(i)] As already mentioned  the $u$-step, defined through \eqref{AAA-ADPMM:Step1}, reduces 
				to the computation of the proximal mapping of the function $h$. Thus, this step can be 
				efficiently computed when the proximal map of $h$ is accessible, \ie via an and explicit 
				formula or via simple computations, see for instance, \cite{LT14,BST2014,BH2016} for 
				interesting examples.
    			\item[(ii)] The second subproblem, namely the $x$-step, is more involved. Let us discuss two 
    				protocols for solving this step approximately. For simplicity, suppose that $f_{0} \equiv 
    				0$. Then, the step \eqref{AAA-ADPMM:Step2} reduces to solve an {\em unconstrained Nonlinear 
    				Least Squares problem}, NLS for short. Therefore, the proposed Lagrangian methodology which 
    				allows to reduce the very general constrained nonlinear optimization model (CM) to solving 
    				sequentially unconstrained NLS subproblems, provides interesting future research avenues, 
    				whereby fundamental methods of NLS could be considered and exploited to analyze inexact 
    				variants. Indeed, NLS problems are  central in scientific computation, and even though these 
    				are nonconvex problems, there exist two well-known fundamental methods: Gauss-Newton and 
    				Levenberg-Marquardt, including many of their variants, which address this key computational 
    				problem within a very large body of literature, see \eg \cite{B1996-B,C2009-B}; see also the 
    				interesting work \cite{NLS}, where SDP relaxations are shown to find global solutions of 
    				some unconstrained NLS of polynomial type. Another approach to tackle the $x$-step is to 
    				approximate it through convex subproblems, which can then be efficiently solved. For this, 
    				we refer the reader to Section \ref{other} where we give further insights into this 
    				question, and we also introduce a new and easily implementable version of {\bf ALBUM 2} for 
    				(CM-L) problems.
    		\end{itemize}
	\end{remark}
	
\subsection{Convergence results for {\bf ALBUM 1} and {\bf ALBUM 2}.}
	To apply our main results (\cf Section \ref{Sec:ALBUM}), as previously explained, we first need to
	verify that \emph{joint minimization} and \emph{alternating minimization} satisfy the two conditions of
	Definition \ref{D:LagAlg}, \ie they are Lagrangian algorithmic maps. Recall that following our
	notations, for a given point $\xi := \xi^{k}$ at iteration $k$, the next point $\xi^{+}$ stands for
	$\xi^{k + 1}$.
	\begin{itemize}
		\item {\bf ALBUM 1} -- {\em Joint Minimization}
		
			From the choice of $\AAA_{\rho}$ (see (\ref{AAA-PMM})) we immediately get
			\begin{equation*}
    			\Laug_{\rho}\left(x^{+} , u^{+} , y\right) + \frac{\mu}{2}\norm{x^{+} - x}^{2} \leq \Laug_{\rho}
    			\left(x , u , y\right),
			\end{equation*}
			showing that Definition \ref{D:LagAlg}(i) holds true with $a = \mu$. Moreover, we also obtain
			\begin{equation}\label{optl}
        			\left(0 , 0\right) \in \left(\nabla_{x} \Laug_{\rho}\left(x^{+} , u^{+} , y\right) + \mu
        			\left(x^{+} - x\right) , \partial_{u} \Laug_{\rho}\left(x^{+} , u^{+} , y\right)\right),
			\end{equation}
			hence it follows that Definition \ref{D:LagAlg}(ii) immediately holds true with $b = \mu$.	
		\item {\bf ALBUM 2} -- {\em Alternating Minimization}
		
			Thanks to the choice of $\AAA_{\rho}$, we get from \eqref{AAA-ADPMM:Step1} that
			$\Laug_{\rho}\left(x , u^{+} , y\right) \leq \Laug_{\rho}\left(x , u , y\right)$ and from
			\eqref{AAA-ADPMM:Step2} we get that $\Laug_{\rho}\left(x^{+} , u^{+} , y\right) + \frac{\mu}{2}
			\norm{x^{+} - x}^{2} \leq \Laug_{\rho}\left(x , u^+ , y\right)$. Combining both inequalities
			shows that Definition \ref{D:LagAlg}(i) holds true with $a = \mu$. Moreover, as before it also
			follows immediately that Definition \ref{D:LagAlg}(ii) holds true with $b = \mu$.	
	\end{itemize}
\medskip
		
	We will now show that both {\bf ALBUM 1} and {\bf ALBUM 2} generate Lagrangian sequences $\Seq{z}{k}$.
	To this end we have to verify that conditions {\bf C3} and {\bf C4} hold true for both schemes.
\medskip

	First, for {\bf ALBUM 1} we obtain from \eqref{AAA-PMM} (\cf \eqref{optl}) that $0 =: v^{k + 1} \in
	\partial_{u} \Laug_{\rho_{k}}\left(x^{k + 1} , u^{k + 1} , y^{k}\right)$, and hence condition {\bf C3}
	holds true with any $c >0$. The next result shows that condition {\bf C3} also holds true for {\bf
	ALBUM 2}.
	\begin{proposition} \label{L:Cond3ADPMM}
		Let $\Seq{z}{k}$ be a sequence generated by {\bf ALBUM 2} which is assumed to be bounded. Then, for
		each $k \in \nn$, there exist a positive constant $c$ and $v^{k + 1} \in \partial_{u}
		\Laug_{\rho_{k}}\left(x^{k + 1} , u^{k + 1} , y^{k}\right)$, such that for all $k \geq \ksta$ we
		have
		\begin{equation*}
			\norm{v^{k + 1}} \leq c\norm{x^{k + 1} - x^{k}}.
		\end{equation*}
    \end{proposition}		
	\begin{proof}
		Since $\Seq{x}{k}$ is bounded, and for each $k \geq \kinfo$, we have that $\nabla F$ is Lipschitz
		continuous on $\Zone$ (by Assumption \ref{AssumptionB}(ii)), it follows that there exists $B > 0$
		such that
     	\begin{equation*}
     		\sup_{k \geq \kinfo} \norm{\nabla F\left(x^{k}\right)} \leq B.
     	\end{equation*}
		From \eqref{AAA-ADPMM:Step1} we get that
		\begin{equation*}
        		0 \in \partial_{u} \Laug_{\rho_{k}}\left(x^{k} , u^{k + 1} , y^{k}\right).
		\end{equation*}
	 	Using the definition of $\Laug_{\rho}$ (see \eqref{D:AugLAc}) we obtain that
		\begin{equation*}
        		\partial_{u} \Laug_{\rho_{k}}\left(x^{k + 1} , u^{k + 1} , y^{k}\right) = \partial_{u}
        		\Laug_{\rho_{k}}\left(x^{k} , u^{k + 1} , y^{k}\right) + \rho_{k}\left(F\left(x^{k}\right) - F
        		\left(x^{k + 1}\right)\right).
		\end{equation*}	
		Therefore, using the inclusion just above, we obtain for all $k \in \nn$ that
		\begin{equation*}
        		v^{k + 1} \equiv \rho_{k}\left(F\left(x^{k}\right) - F\left(x^{k + 1}\right)\right) \in
        		\partial_{u} \Laug_{\rho_{k}}\left(x^{k + 1} , u^{k + 1} , y^{k}\right),
		\end{equation*}		
		and
		\begin{equation*}
       		\norm{v^{k + 1}} = \rho_{k}\norm{F\left(x^{k + 1}\right) - F\left(x^{k}\right)} \leq
       		\rho_{\ksta}B\norm{x^{k + 1} - x^{k}},
		\end{equation*}	
		where the last inequality follows from the Mean Value Theorem\footnote{Recall that $\norm{F\left(u
		\right) - F\left(v\right)} \leq \sup_{\theta \in [0 , 1]} \norm{\nabla F\left(v + \theta\left(u - v
		\right)\right)}\norm{u - v}$, \cite[p. 69]{OR1970-B}} and the fact that $\rho_{k} \leq \rho_{\ksta}$
		for all $k \geq \ksta$ (see Lemma \ref{L:SufficentDecrease}). This proves that condition {\bf C3}
		holds true with $c = \rho_{\ksta}B$. 
	\end{proof}
\medskip

	Having established that the three conditions {\bf C1}, {\bf C2} and {\bf C3} of the basic methodology
	hold, to apply our main convergence results to {\bf ALBUM 1} and {\bf ALBUM 2}, it remains to verify the
	validity of the condition {\bf C4} for $h$. This is done next.
	\begin{proposition} \label{P:ALBUM12C4}
		Let $\Seq{z}{k}$ be a sequence generated by either {\bf ALBUM 1} or {\bf ALBUM 2}, which is assumed
		to be bounded. Let ${\bar z}$ be a limit point of a subsequence $\left\{ z^{k} \right\}_{k \in{\cal
		K}}$ of $\Seq{z}{k}$, then we have that $\limsup_{k \in {\cal K}\subset \nn} h\left(u^{k}\right)
		\leq h\left({\bar u}\right)$.
	\end{proposition}
	\begin{proof}
		The sequence $\Seq{z}{k}$ is bounded and therefore there exists a subsequence $\left\{ z^{m_{k}}
		\right\}_{k \in \nn}$ which converges to ${\bar z} = \left({\bar x} , {\bar u} , {\bar y}\right)$.
\medskip

		For {\bf ALBUM 1}: from the $x$-step  we have for all $k \geq \ksta$ that
    		\begin{equation*}
    			\Laug_{\rho_{\ksta}}\left(x^{k + 1} , u^{k + 1} , y^{k}\right) + \frac{\mu}{2}\norm{x^{k + 1} -
    			x^{k}}^{2} \leq \Laug_{\rho_{\ksta}}\left({\bar x} , {\bar u} , y^{k}\right) + \frac{\mu}{2}
    			\norm{{\bar x} - x^{k}}^{2}.
    		\end{equation*}
    		We now substitute $k$ by $m_{k} - 1$ and obtain from the definition of $\Laug_{\rho}$ (see
    		\eqref{D:AugLAc}) that	
    		\begin{align}
    			f_{0}\left(x^{m_{k}}\right) + h\left(u^{m_{k}}\right) & + \act{y^{m_{k} - 1} , F\left(x^{m_{k}}
    			\right) - F\left({\bar x}\right)} + \act{y^{m_{k} - 1} , {\bar u} - u^{m_{k}}} \nonumber \\
    			& + \frac{\rho_{\ksta}}{2}\norm{F\left(x^{m_{k}}\right) - u^{m_{k}}}^{2}\leq f_{0}\left({\bar
    			x}\right) + h\left({\bar u}\right) + \frac{\rho_{\ksta}}{2}\norm{F\left({\bar x}\right) -
    			{\bar u}}^{2} \notag \\
    			& + \frac{\mu}{2}\norm{{\bar x} - x^{m_{k} - 1}}^{2}. \label{h1}
    		\end{align}
  		Likewise, for {\bf ALBUM 2}, from the $u$-step (see \eqref{AAA-ADPMM:Step1}), we have for all $k
  		\geq \ksta$ that
    		\begin{equation*}
    			\Laug_{\rho_{\ksta}}\left(x^{k} , u^{k + 1} , y^{k}\right)\leq \Laug_{\rho_{\ksta}}\left(x^{k} ,
    			{\bar u} , y^{k}\right).
    		\end{equation*}
    		We now substitute $k$ by $m_{k} - 1$ and obtain from the definition of $\Laug_{\rho}$ (see
    		\eqref{D:AugLAc}) that
    		\begin{equation}\label{h2}
    			h\left(u^{m_{k}}\right) + \act{y^{m_{k} - 1} , u^{m_{k}} - {\bar u}} + \frac{\rho_{\ksta}}{2}
    			\norm{F\left(x^{m_{k} - 1}\right) - u^{m_{k}}}^{2} \leq h\left({\bar u}\right) +
    			\frac{\rho_{\ksta}}{2}\norm{F\left(x^{m_{k} - 1}\right) - {\bar u}}^{2}.
    		\end{equation}
    		For each of the just derived inequalities \eqref{h1} and \eqref{h2}, letting $k$ goes to $\infty$
    		and using the continuity of $f_{0}$ and $F$ (see Assumption \ref{AssumptionB}(ii) and (iii)),
    		together with Proposition \ref{P:WeakConv} (for the case of \eqref{h1}) yields in both cases that
    		\begin{equation*}
    			\limsup_{k \rightarrow \infty} h\left(u^{m_{k}}\right) \leq h\left({\bar u}\right),
    		\end{equation*}
		and the proof is completed.
    \end{proof}
\medskip

	To summarize at this point, we have therefore shown that the two main schemes {\bf ALBUM 1} and {\bf 
	ALBUM 2} produce \emph{Lagrangian sequences} and hence our convergence results Theorems \ref{T:SubConv} 
	and \ref{T:GlobConv} are applicable. Observe that we do not only prove that these well-known methods 
	converge in the absence of convexity for the general nonlinear composite model (CM), we also show how to 
	apply them under weak assumptions through the use of a new adaptive regime.

\subsection{Towards implementable variants of {\bf ALBUM}.} \label{other}
	To further illustrate the potential benefits and generality of our approach we now consider further 
	specific instances and variants of {\bf ALBUM} under other relevant assumptions on data information 
	which occur in many interesting applications. This allows us to extend some recent results in the 
	literature and even to propose a new scheme.
\medskip

	{\bf The classical method of alternating direction of multipliers (ADM).} Consider the limiting case of 
	{\bf ALBUM 2} obtained with $\mu \equiv 0$. We recover the classical Alternating Direction of 
	Multipliers (ADM) \cite{GLT1989}. Under the additional assumption that the augmented Lagrangian $x 
	\rightarrow \Laug_{\rho}\left(x , u , y\right)$ is $\sigma$-strongly convex, for any fixed $u , y \in 
	\rr^{m}$, we can obtain global convergence of the ADM to critical points of the {\em nonlinear} 
	nonconvex composite model (CM). Indeed, in this case {\bf ALBUM 2} yields (recall 
	\eqref{AAA-ADPMM:Step1} and \eqref{AAA-ADPMM:Step2}) that
	\begin{equation} \label{ALBUM2SC:Step1}
		\Laug_{\rho}\left(x , u^{+} , y\right) \leq \Laug_{\rho}\left(x , u , y\right),
	\end{equation}
	and
	\begin{equation} \label{ALBUM2SC:Step2}
		\nabla_{x} \Laug_{\rho}\left(x^{+} , u^{+} , y\right) = 0 .
	\end{equation}
	Now, by the $\sigma$-strong convexity of $x \rightarrow \Laug_{\rho}\left(x , u^{+} , y\right)$ together 
	with \eqref{ALBUM2SC:Step2} we have that
	\begin{equation*}
    		\Laug_{\rho}\left(x^{+} , u^{+} , y\right) + \frac{\sigma}{2}\norm{x^{+} - x}^{2} \leq \Laug_{\rho}
    		\left(x , u^{+} , y\right),
	\end{equation*}
	and hence from \eqref{ALBUM2SC:Step1} it follows that Definition \ref{D:LagAlg}(i) holds true with $a =
	\sigma$. Moreover, we also get that $\norm{\nabla_{x} \Laug_{\rho}\left(x^{+} , u^{+} , y\right)} = 0
	\leq b\norm{x^{+} - x}$, showing that Definition \ref{D:LagAlg}(ii) immediately holds true with any $b > 
	0$. Now it is trivial to see that the proofs of conditions {\bf C3} and {\bf C4} as done for the case $
	\mu > 0$ for {\bf ALBUM 2} remain valid for the case $\mu = 0$. Thus our convergence results apply, and 
	extend the recent result \cite[Theorem 4]{LP2015}, which uses the same assumption on the Lagrangian, but 
	was valid only for the linear case (\ie $F\left(x\right) \equiv Fx$). Furthermore, for the  linear case  
	with a matrix $F$  full row rank, we have $\gamma =\sqrt{\lambda_{\min}(FF^{T})}>0$, and since $a = 
	\sigma$ and $b$ can be any positive number, (\eg we can set $b = 1$), we immediately obtain the 
	threshold value for $\rho$ (see \eqref{rholin} in Remark \ref{r:lin}) that warrant our convergence 
	results:
	\begin{equation*}
		\rho > \frac{4 ((L(f_{0}) + 1)^{2} +1)}{\sigma\lambda_{\min}(FF^{T})}.
	\end{equation*}
\medskip

	{\bf Tractable convex subproblems for ALBUM 2.} In relation to Remark \ref{tract}, we focus on the 
	tractability of the $x$-step (as already mentioned, the $u$-step is easier for any proximable $h$). We 
	illustrate here a specific but fundamental aspect of our family of methods through the important case of 
	{\bf ALBUM 2}. In addition to the standing assumptions, we assume that $f_{0}$ is $C^{2}$ with Lipschitz 
	continuous gradient (for simplicity) and $F$ is linear (so that the information zone is the whole space, 
	\cf Remark \ref{r:info}). The constant $\rho > 0$ can thus be determined. We observe that for fixed 
	couple $\left(u , y\right)$, the function $\Laug_{\rho}\left(\cdot , u , y\right)$ is $C^{2}$ whenever 
	$u$ is in $\dom h$ and that its Hessian matrix is given by $x \rightarrow \nabla^{2} f_{0}\left(x\right) 
	+ \rho F^{T}F$. As a consequence of the Lipschitz continuity assumption of $f_{0}$ we have that:
	\begin{equation}
		\sup _{(x , u , y) \in \rr^{n} \times \rr^{m} \times \dom{h}} \norm{\nabla_{x}^{2} \Laug_{\rho}
		\left(x , u , y\right)} \leq L(f_{0}) + \rho\lambda_{\max}(FF^{T}).
	\end{equation}
	Thus, with $\mu = L(f_{0}) + \rho\lambda_{\max}(FF^{T})$,  the $x$-step in {\bf ALBUM 2} consists in 
	minimizing a {\em convex function} $x \rightarrow \Laug_{\rho}\left(x , u^{k + 1} , y^{k}\right) + 
	\left(\mu/2\right)\norm{x - x^{k}}^{2}$ with {\em known Lipschitz continuous gradient}.
\medskip

	{\bf Solving general semi-algebraic feasibility problems with ALBUM 2.} The specialization of {\bf ALBUM 
	2} to the general feasibility problem described in Example \ref{feas} provides a new parallel projection  
	method; the details of the easy derivation of the corresponding steps in this case are left to the 
	reader. In view of our general results, the penalty parameter $\rho > 0$ can be determined and no other 
	assumption than semi-algebraicity of the subsets $S_{i}$, $i = 1 , 2, \ldots , p$, is necessary to 
	obtain global convergence of the methods (under our classical boundedness assumptions)
\medskip

	{\bf A simple explicit algorithm: Proximal Linearized Alternating Minimization.} We consider here a 
	proximal linearized instance of {\bf ALBUM 2} with proven global convergence results which seems to be 
	new in the literature for the nonconvex composite model. Our setting here is confined to the particular, 
	yet interesting and important case, where in the model (CM):
\medskip

	\begin{itemize}
		\item The function $f_{0}$ has an $L(f_{0})$-Lipschitz continuous gradient on $\rr^{n}$.
		\item The mapping $F$ is linear, namely  $F\left(x\right) \equiv Fx$ for all $x \in \rr^{n}$, for 
			some matrix $F \in \rr^{n \times m}$ with full row rank.
  	\end{itemize}

  	Furthermore, we additionally assume that $\kappa(FF^{T}) < 2$, where $\kappa(A)$ denotes the condition 
  	number of a square matrix $A$, namely the ratio $\lambda_{\max}(A) / \lambda_{\min}(A)$. 
\medskip

  	Note that this assumption  always holds true whenever $FF^{T}$ or $F^{T}F$ is the identity matrix, which 
  	often occurs in applications, \eg in some problems in signal recovery \cite{BBC2011}.
\medskip

 	Recall (cf. Remark \ref{r:info}) that under the above hypothesis on the problem's data, Assumption 
 	\ref{AssumptionB} holds with ${\cal Z} \equiv \rr^{n}$, and we also have that $\gamma = 
 	\sqrt{\lambda_{\min}(FF^T)} > 0$. The augmented Lagrangian in this case reads (\cf \eqref{D:AugLAc}), 
 	for $\rho > 0$, as follows
	\begin{equation*}
		\Laug_{\rho}\left (x , u , y\right) := f_{0}\left(x\right) + h\left(u\right) + \act{y , Fx
		 - u} + \frac{\rho}{2}\norm{Fx - u}^{2}.
	\end{equation*}
	We then consider approximating the $x$-step in {\bf ALBUM 2} (leaving the $u$-step untouched) through 
	the following scheme:
\medskip

	\begin{itemize}
		\item {\bf ALBUM 3} -- {\em Proximal Linearized Alternating Minimization}
	\end{itemize}
	\begin{align}
		u^{k + 1} & \in \argmin_{u} \Laug_{\rho_{k}}\left(x^{k} , u , y^{k}\right), \label{ALBUM3:1} \\
		x^{k + 1} & \in \argmin_{x} \left\{ \act{x - x^{k} , \nabla_{x} \Laug_{\rho_{k}}\left(x^{k} , u^{k +
		1} , y^{k}\right)} + \frac{\mu}{2}\norm{x - x^{k}}^{2} \right\}, \quad (\mu > 0). \label{ALBUM3:2}
	\end{align}
	Thus, the $x$-step consists of first linearizing  the augmented Lagrangian around a given point and 
	adding a proximal term, which is a common strategy to generate a simpler approximate step (see \eg 
	\cite{BST2014}), and hence \eqref{ALBUM3:2} is nothing else but {\em one  shot} of an explicit gradient 
	step for minimizing $\Laug_{\rho_{k}}\left (x , u^{k+1} , y^{k}\right)$, with an easy explicit formula.
\medskip

	To apply the convergence results of Section \ref{Sec:ALBUM}, we first need to verify that the 
	corresponding algorithmic map ${\cal A}_{\rho}$ of {\bf ALBUM 3} satisfies the two conditions of 
	Definition \ref{D:LagAlg} ,\ie is a Lagrangian algorithmic map. For that purpose, first note that given 
	couple $\left(u , y\right)$, the gradient of $\Laug_{\rho}\left(x , u , y\right)$ with respect to $x$, 
	is the mapping $x \rightarrow \nabla f_{0}\left(x\right) + F^{T}y + \rho F^{T}\left(Fx - u\right)$, 
	which is a $L$-Lipschitz continuous mapping, with $L:= L(f_{0}) + \rho\norm{F}^{2}$. Invoking the well 
	known Descent Lemma, it follows that condition {\bf C1} holds with $a = \mu - L/2$.  However, observe 
	that contrary to ALBUM 1 and 2, the constant $a$ depends on $\rho$ through $L$, and $a>0$ will be 
	warranted thanks to Lemma \ref{L:ALBUM3Rho} given below.
\medskip

	Next, using the steps of the corresponding algorithmic map $\AAA_{\rho}$, together with the fact that 
	$f_{0}$ admits an $L(f_{0})$-Lipschitz continuous gradient, one easily verifies that for any $k \geq 0$,
	\begin{align} \label{L:C2ALBUM3}
        	\norm{\nabla_{x} \Laug_{\rho}\left(x^{+} , u^{+} , y\right)} & \leq \norm{\nabla_{x} \Laug_{\rho}
        	\left(x^{+} , u^{+} , y\right) - \nabla_{x} \Laug_{\rho}\left(x , u^{+} , y\right)} + 
        	\norm{\nabla_{x} \Laug_{\rho}\left(x , u^{+} , y\right)} \nonumber \\
        	& \leq \left(L(f_{0}) + \rho\norm{F}^{2} + \mu\right)\norm{x^{+} - x}.
  	\end{align}
	This shows that condition {\bf C2} holds true with $b = L(f_{0}) + \rho\norm{F}^{2} + \mu$. In addition, 
	condition {\bf C3} is immediate, since here the optimality condition of the $u$-step (see 
	\eqref{ALBUM3:1}) implies that
	\begin{equation*}
		\partial_{u} \Laug_{\rho_{k}}\left(x^{k + 1} , u^{k + 1} , y^{k}\right) \ni v^{k + 1} = \rho F
		\left(x^{k + 1} - x^{k} \right) \; \Rightarrow \; \norm{v^{k + 1}} \leq \rho\norm{F}\norm{x^{k + 1}
		- x^{k}},
	\end{equation*}
 	showing that condition {\bf C3} holds with $c = \rho\norm{F}$.
\medskip

	Finally, since the $u$-step in {\bf ALBUM 3} is identical to the one in {\bf ALBUM 2}, the statement and 
	the proof of Proposition \ref{P:ALBUM12C4} holds in this case with the same proof (see only the part 
	that related to {\bf  ALBUM 2}), and hence condition {\bf C4} holds true in this case too.
\medskip

	Despite the fact that conditions {\bf C1--C4} are satisfied it is important to realize that our general 
	theorem does not apply at this stage because both $a$ and $b$ depend on $\rho$ and may become negative 
	if $\rho$ is too large. In order to circumvent this difficulty and obtain the general convergence of the 
	scheme (as in Theorems \ref{T:SubConv} and \ref{T:GlobConv}), it suffices to guarantee a sufficient 
	descent of the Lyapunov function $\EE_{\beta}$. For this we need that \eqref{descond} holds true (see 
	Remark~\ref{liap}(b)) for a couple of well chosen $\mu$ and $\rho$, that is,
    	\begin{equation} \label{ParaDescGene}
			\frac{a}{2} - \frac{d_{1} + d_{2}}{\rho} > 0.
	\end{equation}
	For that purpose let us first observe that a stronger version of Lemma \ref{L:DualSequence} can be 
	derived. Just follow the same proof  by exploiting the linearity of $F$, and note that the boundedness 
	assumption on the sequence of multipliers $\Seq{y}{k}$ in not anymore needed in that case. We leave the 
	details to the reader, and record this result below.
	\begin{lemma} \label{L:DualSequenceLinALBUM3}
        Let $\Seq{z}{k}$ be a Lagrangian sequence. Then, the following inequality holds
        true for any $k \geq 0$,
		\begin{equation} \label{L:DualSequenceLinALBUM3:0}
        		\norm{y^{k + 1} - y^{k}}^{2} \leq d_{1}\norm{x^{k + 1} - x^{k}}^{2} + d_{2}\norm{x^{k} -x^{k
        		- 1}}^{2},
		\end{equation}
		where
		\begin{equation} \label{L:DualSequenceLinALBUM3:1}
        		d_{1} = \frac{2\norm{{\cal M}}^{2}}{\lambda_{\min}(FF^{T})}, \quad d_{2} =
        		\frac{2\left(L(f_{0}) + \norm{{\cal M}}\right)^{2}}{\lambda_{\min}(FF^{T})},
		\end{equation}
		and ${\cal M} := \mu I_{n} - \rho F^{T}F$.
    \end{lemma}
	Equipped with this result, we now show that we can find positive constants $\rho$ and $\mu$ in terms of 
	the problem's data so that \eqref{ParaDescGene} holds, and hence our convergence results for {\bf ALBUM 
	3}: Theorems \ref{T:SubConv} and \ref{T:GlobConv} with semi-algebraic data, apply. 
	\begin{lemma}[Determining threshold value for $\rho$] \label{L:ALBUM3Rho}
		Let $F : \rr^{n} \rightarrow \rr^{m}$ be a linear mapping for which $\kappa(FF^{T}) < 2$. Let
		$\Seq{z}{k}$ be a sequence generated by {\bf ALBUM 3}. Then, there exists a constant ${\bar
		\rho}$ such that \eqref{ParaDescGene} holds for any $\rho > {\bar \rho}$, and with $\mu \in
		\left(\mu_{1} , \mu_{2}\right)$ for some $\mu_{1} , \mu_{2} >0$, where both ${\bar \rho}$, $\mu_{1}$
		and $\mu_{2}$ are given in terms of the problem's data $L(f_{0})$ and $\gamma$.
	\end{lemma}
	\begin{proof}
		For convenience we denote $\ell := L(f_{0})$. Using Lemma \ref{L:DualSequenceLinALBUM3} and the fact
		that $a = \mu - \left(\ell + \rho\norm{F}^{2}\right)/2$, in order to satisfy \eqref{ParaDescGene},
		we need to find $\rho > 0$ and $\mu > 0$ such that
   		\begin{equation} \label{L:ALBUM3Rho:1}
			\frac{\mu - \frac{\ell + \rho\norm{F}^{2}}{2}}{2} - \frac{2\norm{{\cal M}}^{2} + 2\left(\ell + 
			\norm{{\cal M}}\right)^{2}}{\rho\gamma^{2}} > 0.
		\end{equation}
		Rewriting this inequality yields the following equivalent one
		\begin{equation*}
			16\norm{{\cal M}}^{2} + \rho\gamma^{2}\left(\ell + \rho\norm{F}^{2} - 2\mu\right) + 8\ell^{2} + 
			16\ell\norm{{\cal M}} < 0.
		\end{equation*}
		Since ${\cal M} = \mu I - \rho F^{T}F$, and symmetric we have
		\begin{equation*}
			\norm{{\cal M}} = \lambda_{\max}({\cal M}) = \lambda_{\max}(\mu I - \rho F^{T}F) =
			\lambda_{\max}(\mu I) - \rho\lambda_{\min}(F^{T}F) = \mu - \rho\gamma^{2},
		\end{equation*}
		where the last equality uses the fact $\lambda_{\min}(F^{T}F)=\lambda_{\min}(FF^{T}) = \gamma^2$.

		Therefore, defining $t: = \mu - \rho\gamma^{2}\equiv \norm{{\cal M}}$, and rearranging terms, the 
		above inequality reduces to show that
		\begin{equation} \label{psi}
			\psi\left(t\right):= 16t^{2} - 2\left(\rho\gamma^{2} - 8\ell\right)t + \rho\gamma^{2}\left(\ell 
			+ \rho\norm{F}^{2} - 2\rho\gamma^{2}\right) + 8\ell^{2} 	< 0.
		\end{equation}	
		Computing the (reduced) discriminant $\Delta_{\psi}$ of the above quadratic function $\psi
		\left(\cdot\right)$ yields
		\begin{equation*}
			\Delta_{\psi} := \left(\rho\gamma^{2} - 8\ell\right)^{2} - 16\left(\rho\gamma^{2}\left(\ell + 
			\rho\norm{F}^{2} - 2\rho\gamma^{2}\right) + 8\ell^{2}\right)  = \rho^{2}\gamma^{2}\eta - 32\rho
			\gamma^{2}\ell - 64\ell^{2},
		\end{equation*}
		where thanks to our assumption  $\kappa(FF^{T}) < 2$, we have $\eta := \left(33\gamma^{2} - 
		16\norm{F}^{2}\right)>0$. Therefore, \eqref{psi} holds (and hence so does \eqref{L:ALBUM3Rho:1}), if 
		$\Delta_{\psi} > 0$ and $t_{1} < t < t_{2}$ where $t_{1}$ and $t_{2}$ are the zeroes of $\psi\left(t
		\right)$. Some algebra then shows that the latter is satisfied with
		\begin{equation*}
			\rho > {\bar \rho}:= \frac{8\ell}{\eta\gamma}\left(2\gamma + \sqrt{4\gamma^{2} + \eta}\right),
		\end{equation*}
		and
		\begin{equation}  \label{L:ALBUM3Rho:2}
			t_{1} \equiv \frac{\left(\rho\gamma^{2} - 8\ell\right) - \sqrt{\Delta_{\psi}}}{16} < t <  
			\frac{\left(\rho\gamma^{2} - 8\ell\right) + \sqrt{\Delta_{\psi}}}{16} \equiv t_{2}.
		\end{equation}
		Moreover, since $\norm{{\cal M}} = \mu - \rho\gamma^{2} = t$, we must have $t \geq 0$, and indeed it 
		is easy to check that $t_{1} >0$. Using the relation $\mu = t + \rho\gamma^{2}$, we can rewrite  
		\eqref{L:ALBUM3Rho:2} as follows
		\begin{equation*}
			\mu_{1} \equiv \frac{\left(17\rho\gamma^{2} - 8\ell\right) - \sqrt{\Delta_{\psi}}}{16} < \mu <  
			\frac{\left(17\rho\gamma^{2} - 8\ell\right) + \sqrt{\Delta_{\psi}}}{16} \equiv \mu_{2}.
		\end{equation*}		
		and the proof is completed. 
	\end{proof}
	
\section*{Acknowledgments.}
	The research of J\'{e}r\^{o}me Bolte is sponsored by the Air Force Office of Scientific Research, Air
	Force Material Command, USAF, under grant number FA9550-15-1-0500 \& the FMJH Program Gaspard Monge in
	optimization and operations research. The research of Shoham Sabach was partially supported by the
	German Israel Foundation, GIF Grant G-1243-304.6/2014. The research of Marc Teboulle was partially
	supported by the Israel Science Foundation, ISF Grants 998/12 and 1844/16, and the German Israel
	Foundation, GIF Grant G-1243-304.6/2014.

\bibliographystyle{informs2014}	

\end{document}